\documentclass{article}

\usepackage{xparse}
\usepackage{amsmath}
\usepackage{amsthm}
\usepackage{amsfonts}
\usepackage{amssymb}
\usepackage{enumitem}
\usepackage{parskip}
\usepackage{lmodern}
\usepackage{scalerel}
\usepackage{thmtools}
\usepackage{hyperref}
\hypersetup{bookmarksopen, bookmarksnumbered, hidelinks}

\usepackage[backend=biber,citestyle=numeric, bibstyle=numeric]{biblatex}

\NewBibliographyString{toappear}
\DefineBibliographyStrings{english}{%
  toappear = {to appear},
}

\renewbibmacro*{in:}{%
  \iffieldundef{pubstate}
    {}
    {\printfield{pubstate}%
     \setunit{\addspace}%
     \clearfield{pubstate}}%
  \printtext{%
    \bibstring{in}\intitlepunct}}
\DeclareSymbolFont{bbold}{U}{bbold}{m}{n}
\DeclareSymbolFontAlphabet{\mathbbold}{bbold}

\DeclareFontFamily{U}{mathx}{\hyphenchar\font45}
\DeclareFontShape{U}{mathx}{m}{n}{
      <5> <6> <7> <8> <9> <10>
      <10.95> <12> <14.4> <17.28> <20.74> <24.88>
      mathx10
      }{}
\DeclareSymbolFont{mathx}{U}{mathx}{m}{n}
\DeclareFontSubstitution{U}{mathx}{m}{n}

\makeatletter
\newcommand{\bigdoublevee}{\big@doubleop{\bigvee}}
\newcommand{\bigdoublewedge}{\big@doubleop{\bigwedge}}
\newcommand{\big@doubleop}[1]{%
  \DOTSB\mathop{\mathpalette\big@doubleop@aux{#1}}\slimits@
}

\newcommand\big@doubleop@aux[2]{%
  \sbox\z@{$\m@th#1#2$}%
  \makebox[1.35\wd\z@][s]{$\m@th#1#2\hss#2$}%
}
\makeatother

\makeatletter
\newcommand{\dotminus}{\mathbin{\text{\@dotminus}}}

\newcommand{\@dotminus}{%
  \ooalign{\hidewidth\raise1ex\hbox{.}\hidewidth\cr$\m@th-$\cr}%
}
\makeatother

\declaretheorem[parent=section]{theorem}
\declaretheorem[sibling=theorem]{lemma}
\declaretheorem[sibling=theorem]{corollary}

\declaretheorem[sibling=theorem, style=definition]{definition}
\declaretheorem[sibling=theorem]{fact}
\declaretheorem[sibling=theorem]{proposition}
\declaretheorem[sibling=theorem]{question}
\declaretheorem[sibling=theorem, style=definition]{examples}

\newcommand{\mc}[1]{\mathcal{#1}}
\newcommand{\mb}[1]{\mathbb{#1}}

\newcommand{\vk}[1]{\overline{#1}}

\DeclareMathOperator{\st}{\colon\,}

\DeclareMathOperator{\sgn}{sgn}

\newcommand{\norm}[1]{\left\lVert#1\right\rVert}

\renewcommand{\epsilon}{\varepsilon}

\DeclareMathOperator{\card}{card}

\newcommand{\conj}[1]{\overline{#1}}
\DeclareMathOperator{\pipeseparator}{|}

\NewDocumentCommand{\cgen}{omm}{%
  \IfNoValueTF{#1}
  {\mathrm{C}_{\mathrm{gen}}^{\mbox{$*$}} (#3;#2)}
  {\mathrm{C}_{\mathrm{gen}}^{\mbox{$*$}}\mathopen{#1(}#3;#2\mathclose{#1)}}
}%

\NewDocumentCommand{\kcgen}{om}{%
  \IfNoValueTF{#1}
  {\cgen{\mb{K}}{#2}}
  {\cgen[#1]{\mb{K}}{#2}}
}%

\NewDocumentCommand{\ugcsa}{omm}{%
  \IfNoValueTF{#1}
  {\mathrm{C}_{\mathrm{uni}}^{\mbox{$*$}} (#3;#2)}
  {\mathrm{C}_{\mathrm{uni}}^{\mbox{$*$}}\mathopen{#1(}#3;#2\mathclose{#1)}}
}%

\NewDocumentCommand{\rgcsa}{omm}{%
  \IfNoValueTF{#1}
  {\mathrm{C}^{\mbox{$*$}}_{\mathrm{red}} (#3;#2)}
  {\mathrm{C}^{\mbox{$*$}}_{\mathrm{red}}\mathopen{#1(}#3;#2\mathclose{#1)}}
}%

\NewDocumentCommand{\urep}{ommm}{%
  \IfNoValueTF{#1}
  {\mathrm{C}^{\mbox{$*$}} \langle#3\pipeseparator#4;#2\rangle}
  {\mathrm{C}^{\mbox{$*$}}_{#2}\mathopen{#1\langle}#3\pipeseparator#4\mathclose{#1\rangle}}
}%

\NewDocumentCommand{\gurep}{omm}{%
  \IfNoValueTF{#1}
  {\mathrm{Grp}\langle#2\pipeseparator#3\rangle}
  {\mathrm{Grp}\mathopen{#1\langle}#2\pipeseparator#3\mathclose{#1\rangle}}
}%

\NewDocumentCommand{\blo}{omm}{%
  \IfNoValueTF{#1}
  {B(#3;#2)}
  {B\mathopen{#1(}#3;#2\mathclose{#1)}}
}%

\NewDocumentCommand{\kblo}{om}{%
  \IfNoValueTF{#1}
  {\blo{\mb{K}}{#2}}
  {\blo[#1]{\mb{K}}{#2}}
}%

\DeclareMathOperator{\vspanop}{span}
\NewDocumentCommand{\vspan}{omm}{%
  \IfNoValueTF{#1}
  {\vspanop_{#2}(#3)}
  {\vspanop_{#2}\mathopen{#1(}#3\mathclose{#1)}}
}%

\NewDocumentCommand{\cts}{omm}{%
  \IfNoValueTF{#1}
  {C(#3;#2)}
  {C\mathopen{#1(}#3;#2\mathclose{#1)}}
}%

\NewDocumentCommand{\ctsvai}{omm}{%
  \IfNoValueTF{#1}
  {C_0(#3;#2)}
  {C_0\mathopen{#1(}#3;#2\mathclose{#1)}}
}%

\NewDocumentCommand{\ctsconj}{omm}{%
  \IfNoValueTF{#1}
  {C(#3, #2)}
  {C\mathopen{#1(}#3,#2\mathclose{#1)}}
}%

\NewDocumentCommand{\ctsconjvai}{omm}{%
  \IfNoValueTF{#1}
  {C_0(#3, #2)}
  {C_0\mathopen{#1(}#3,#2\mathclose{#1)}}
}%

\newcommand{\contr}[2]{\mathcal{F}(#2;#1)}
\DeclareMathOperator{\Iden}{Iden}

\newcommand{\inlinedef}[1]{\textbf{#1}}

\newcommand{\mTitle}{Computable presentations of C*-algebras}
\newcommand{\mAuthorName}{Alec Fox}

\begin{document}

\begin{center}
\textbf{\mTitle}

{\mAuthorName}
\end{center}

\begin{abstract}
We initiate the study of computable presentations of real and complex C*-algebras under the program of effective metric structure theory.
  With the group situation as a model, we develop corresponding notions of recursive presentations and word problems for C*-algebras, and show some analogous results hold in this setting.
  Famously, every finitely generated group with a computable presentation is computably categorical, but we provide a counterexample in the case of C*-algebras.
  On the other hand, we show every finite-dimensional C*-algebra is computably categorical.
\end{abstract}

\section{Introduction}
With his 1882 paper~\cite{Dyck_GruppentheoretischeStudien_1882}, Dyck began the study of presentations of groups in terms of generators and relations, and in doing so laid the foundation for what would become the field of combinatorial group theory.
By the mid-twentieth century, mathematical logic had been incorporated into the field to incredible success.
A success that was, perhaps, most exemplified by the independently proven result of Boone~\cite{Boone_WordProblem_1958} and Novikov~\cite{Novikov_AlgorithmicUnsolvability_1955} of the existence of finitely presented groups with unsolvable word problem.
In that same era, defining work by Fr\"ohlich and Shepherdson in effective field theory~\cite{FrohlichShepherdson_EffectiveProcedures_1956} and by Mal'tsev in effective algebra~\cite{Maltsev_ConstructiveAlgebras_1961}\cite{Maltsev_RecursiveAbelian_1962} established what would later be known as computable structure theory, i.e., the study of the relationship between computability theoretic complexity and mathematical structures.
Although implicit in the 1989 text~\cite{Pour-elRichards_ComputabilityAnalysis_2017} by Pour-El and Richards, it is only within the last decade that a program has emerged to extend computable structure theory to the uncountable structures one might encounter in analysis.
Now known as effective metric structure theory, the program truly began with the work of Melkinov and Nies~\cite{MelnikovNies_ClassificationProblem_2013} on the classification of compact metric spaces and the work of Melnikov~\cite{Melnikov_ComputablyIsometric_2013} on the categoricity of various metric spaces.
Our goal is to apply perspectives and techniques from both combinatorial group theory and effective metric structure theory to C*-algebras.

Complex C*-algebras have become a fixture of modern mathematics, and while real C*-algebras have not received the same attention, they represent a natural class of objects for consideration from the viewpoint of computability theory.
Importantly, the classes of real and complex C*-algebras share similarities with the class of groups.
In particular, C*-algebras can be studied by their presentations in terms of generators and relations, C*-algebras are principally determined by their algebraic structure, and the universal contraction C*-algebras, while not truly free objects, can fulfill some of the same roles as free groups.
Furthermore, any discrete group has corresponding universal and reduced group C*-algebras, so theorems for groups can often serve as a bound on theorems for C*-algebras.

Over the last decade, effective metric structure theory has been developed in the context of metric spaces~\cite{FranklinMcnicholl_DegreesLowness_2020}~\cite{GreenbergEtal_UniformProcedures_2018}, Polish spaces~\cite{BazhenovEtal_ComputableStone_2021}~\cite{Harrison-trainorEtal_ComputabilityPolish_2020}~\cite{HoyrupEtal_DegreeSpectra_2020a}, $\ell^p$ spaces~\cite{Mcnicholl_NoteComputable_2015}~\cite{Mcnicholl_ComputableCopies_2017}~\cite{McnichollStull_IsometryDegree_2019} and $L^p$ spaces~\cite{BrownMcnicholl_AnalyticComputable_2020}~\cite{BrownEtal_ComplexityClassifying_2020}~\cite{ClaninEtal_AnalyticComputable_2019}~\cite{Mcnicholl_ComputingExponent_2020}.
This paper fills a remaining hole by initiating the study of real and complex C*-algebras under the program of effective metric structure theory.

As combinatorial group theory and computable structure theory introduce conflicting terminology, we follow~\cite{Melnikov_ComputableAbelian_2014} in referring to presentations of groups in terms of c.e.\ generators and c.e.\ relations as c.e.\ presentations instead of recursive presentations.
In \autoref{sec3}, we adapt c.e.\ presentations and word problems from groups to C*-algebras and develop the basic theory.
Specifically, we give characterizations for when c.e.\ presentations are actually computable in the sense of computable structure theory, and show that the computability theoretic properties of a presentation can be defined in terms of the word problem.
We also investigate the connection between computable presentations of groups and properties of the word problems of the corresponding universal and reduced group C*-algebras.
In \autoref{sec4}, we describe the relationship between real C*-algebras and their complexifications.
Consequently, we are able to transfer some already established results about $\cts{\mb{R}}{X}$ as a real Banach space or algebra to $\cts{\mb{C}}{X}$ as a complex C*-algebra. Of particular importance, using a result of Melnikov and Ng~\cite{MelnikovNg_ComputableStructures_2016}, we find that, as a complex C*-algebra, $\cts{\mb{C}}{[0,1]}$ is finitely generated and admits a computable presentation but is not computably categorical.
In \autoref{sec5}, we show, on the other hand, that every finite-dimensional C*-algebra is computably categorical.

\section{Preliminaries}
\subsection{C*-algebras}
Although, historically, real C*-algebras were rarely studied in their own right, interest in real C*-algebras has grown as it has become clear that the theory of real C*-algebras can not be subsumed under the theory of complex C*-algebras (see~\cite{MoslehianEtal_SimilaritiesDifferences_2022} for an overview of some differences) and as new applications of real C*-algebras have been found (see~\cite{Rosenberg_StructureApplications_2015} for an overview of some applications).
For our purposes, the framework of real C*-algebras also provides a crucial link between complex C*-algebras and previously established results for real Banach algebras.
In this section, we provide an introduction to real and complex C*-algebras.
More information can be found on real C*-algebras in~\cite{Li_RealOperator_2003} or~\cite{Schroder_KtheoryReal_1993}, and on complex C*-algebras in~\cite{Davidson_AlgebrasExample_1996} or~\cite{Murphy_AlgebrasOperator_2014}.

Throughout the paper, we let $\mb{K}$ denote the real numbers $\mb{R}$ or the complex numbers $\mb{C}$.

\begin{definition}
  A \inlinedef{C*-algebra} over $\mb{K}$ is a Banach *-algebra over $\mb{K}$ which is isometrically *-isomorphic to a norm-closed *-subalgebra of the set of bounded operators $\blo{\mb{K}}{\mc{H}}$ on a Hilbert space $\mc{H}$ over $\mb{K}$.
\end{definition}

We have the following abstract characterizations.
A complex Banach *-algebra $A$ is a complex C*-algebra if and only if it satisfies the C*-axiom $\norm{x}^{2} = \norm{x^{*}x}$ for all $x \in A$.
In the case of real C*-algebras, however, this is no longer enough.
A real Banach *-algebra $A$ is a real C*-algebra if and only if $\norm{x}^{2} \leq \norm{x^{*}x + y^{*}y}$ for all $x,y \in A$.

If $A$ is a unital complex C*-algebra and $x \in A$, then the \inlinedef{spectrum} $\sigma_A(x)$ of $x$ in $A$ is defined to be the set $\{\lambda \in \mb{C} \st \lambda 1 - x \text{ is not invertible} \}$, where $1$ is the unit in $A$.
For any complex C*-algebra $A$, there is a unique unital complex C*-algebra $\tilde{A}$, called the \inlinedef{unitization} of $A$, such that $A$ is an ideal of $\tilde{A}$ and $\tilde{A} / A \cong \mb{C}$.
If $A$ is a nonunital complex C*-algebra and $x \in A$, then the \inlinedef{spectrum} $\sigma_A(x)$ of $x$ in $A$ is defined to be $\sigma_{\tilde{A}}(x)$ where $\tilde{A}$ is the unitization of $A$.

The primary tool connecting real and complex C*-algebras is that of complexification.

\begin{definition}
  Given a real C*-algebra $A$, its \inlinedef{complexification} $A_{c}$ is the complex C*-algebra $A \otimes_{\mb{R}} \mb{C} = A + iA$ equipped with the natural complex *-algebra operations and the induced C*-norm.
\end{definition}
Note $A_{c}$ truly is a complexification of $A$.
Namely, the norm on $A_{c}$ extends the norm on $A$, and $\norm{x + iy} = \norm{x - iy}$ for all $x,y \in A$.
We also have the useful consequence that $\max(\norm{x}, \norm{y}) \leq \norm{x + iy}$ for all $x,y \in A$.

Through the complexification, we can define the spectrum of elements in a real C*-algebra.
If $A$ is a real C*-algebra and $x \in A$, then the \inlinedef{spectrum} $\sigma_{A}(x)$ of $x$ in $A$ is defined to be the spectrum $\sigma_{A_{c}}(x)$ of $x$ in $A_{c}$.
When $A$ is unital, we can further characterize the spectrum by noting that $a + ib \in \sigma_{A}(x)$ if and only if $(x - a)^{2} + b^{2}$ is not invertible in $A$, for $a,b \in \mb{R}$.

\begin{definition}
  Given a complex C*-algebra $A$, a \inlinedef{conjugation} on $A$ is a conjugate-linear map $\tau : A \to A$ such that $\tau(\tau(x)) = x$, $\tau(xy) = \tau(x)\tau(y)$, and $\tau(x^{*}) = \tau(x)^{*}$ for $x,y \in A$. \end{definition}
If $A$ is a real C*-algebra, then there is a natural conjugation $\tau$ on $A_{c}$ given by $\tau(x + iy) = x - iy$ for $x,y \in A$.
We can recover $A$ from $\tau$ as the set of fixed points of $\tau$, $\{z \in A_c \st \tau(z) = z\}$, with the induced operations and norm.
In the same way, every conjugation $\tau$ on a complex C*-algebra $A$ determines a real C*-algebra.

The natural maps between C*-algebras are those that preserve the *-algebraic structure, namely \inlinedef{*-homomorphisms}.
If $\psi : A \to B$ is a real *-homomorphism between real C*-algebras, then $\psi$ extends to a conjugation-preserving complex *-homomorphism $\psi_{c} : A_{c} \to B_{c}$ given by $\psi_{c}(x + iy) = \psi(x) + i\psi(y)$ for $x,y \in A$.
Furthermore, the kernel of $\psi_{c}$ is the complexification of the kernel of $\psi$.
Note any *-homomorphism between C*-algebras is necessarily norm-decreasing.
In particular, any *-isomorphism preserves the norm.

Let $A$ be a C*-algebra over $\mb{K}$.
We say an element $x \in A$ is \inlinedef{self-adjoint} if $x^{*} = x$, and \inlinedef{skew-adjoint} if $x^{*} = -x$.
Every $z \in A$ can be uniquely expressed as $x + y$ where $x$ is self-adjoint and $y$ is skew-adjoint, just let $x = \frac{1}{2}(z + z^{*})$ and $y = \frac{1}{2}(z - z^{*})$.
If $A$ is a complex C*-algebra, then $z$ can also be uniquely expressed as $a + ib$ where $a$ and $b$ are both self-adjoint, just let $a = \frac{1}{2}(z + z^{*})$ and $b = \frac{1}{2i}(z - z^{*})$.

In both the real and complex case, we have a complete classification of the finite-dimensional C*-algebras.

\begin{fact}\label{structure of fd ccsalgs}
  Every finite-dimensional complex C*-algebra is isomorphic to a direct sum of matrix algebras $\bigoplus_{i=1}^k M_{n_i}(\mb{C})$ for some
  $n_1,\ldots,n_k \in \mb{N}$.
\end{fact}

We denote the ring of quaternions by $\mb{H}$.

\begin{fact}\label{structure of fd rcsalgs}
  Every finite-dimensional real C*-algebra is isomorphic to a direct sum of matrix algebras $\bigoplus_{i=1}^k M_{n_i}(\mb{D}_i)$ for some $n_1,\ldots,n_k \in \mb{N}$, where each $\mb{D}_i$ is $\mb{R}$, $\mb{C}$, or $\mb{H}$.
\end{fact}

We also have a complete classification of the abelian C*-algebras.

\begin{fact}\label{accsa}
  Every abelian complex C*-algebra $A$ is isomorphic to $\ctsvai{\mb{C}}{X}$, the space of continuous complex-valued functions on $X$ vanishing at infinity, where $X$ is the set of nonzero complex *-homomorphisms from $A$ to $\mb{C}$.
\end{fact}

If $X$ is a locally compact Hausdorff space and $\theta : X \to X$ is a homeomorphism such that $\theta(\theta(x)) = x$ for $x \in X$, then we let \[\ctsconjvai{\theta}{X} = \{f \in \ctsvai{\mb{C}}{X} \st \conj{f(\theta(x))} = f(x) \text{ for } x \in X\}.\]

\begin{fact}\label{arcsa}
  Every abelian real C*-algebra $A$ is isomorphic to $\ctsconjvai{\theta}{X}$, where $X$ is the set of nonzero real *-homomorphisms from $A$ to $\mb{C}$, and $\theta : X \to X$ is defined by $\theta(x) = \conj{x}$.
\end{fact}



Universal C*-algebras are often difficult to realize concretely, but as we will see further on, their definition in terms of generators and relations grants access to an effective perspective.

Let $\mc{G}$ be a set of noncommuting indeterminates, which we call \textbf{generators}.
Let $\mc{R}$ be a set of \textbf{relations} in $\mc{G}$; specifically, relations of the form $\norm{p(x_{1},\ldots,x_{n})} \leq a$ where $p$ is a *-polynomial over $\mb{K}$ in $n$ noncommuting variables with no constant term, $x_{1},\ldots,x_{n}$ belong to $\mc{G}$, and $a$ is a nonnegative real number.
We also require that for every generator $x \in \mc{G}$ there is a relation of the form $\norm{x} \leq M$ in $\mc{R}$.
We will often follow the convention of writing $p = q$ for the relation $\norm{p - q} \leq 0$.
A \textbf{representation of $(\mc{G}, \mc{R})$} is an assignment of generators $j : \mc{G} \to A$, where $A$ is a C*-algebra over $\mb{K}$, such that $\norm{p(j(x_{1}),\ldots,j(x_{n}))}_{A} \leq a$ for every relation $\norm{p(x_{1},\ldots,x_{n})} \leq a$ in $\mc{R}$.

\begin{definition}
     The \inlinedef{universal C*-algebra} of $(\mc{G},\mc{R})$ over $\mb{K}$ is a C*-algebra $\urep{\mb{K}}{\mc{G}}{\mc{R}}$ over $\mb{K}$, along with a representation $\iota : \mc{G} \to \urep{\mb{K}}{\mc{G}}{\mc{R}}$ of $(\mc{G},\mc{R})$, such that for all representations $j : \mc{G} \to A$ of $(\mc{G},\mc{R})$ there is a unique *-homomorphism $\varphi : \urep{\mb{K}}{\mc{G}}{\mc{R}} \to A$ for which $\varphi(\iota(x)) = j(x)$ for all $x \in \mc{G}$.
\end{definition}
If $\urep{\mb{K}}{\mc{G}}{\mc{R}}$ exists, then it is unique up to isomorphism, and it is generated by $\iota(\mc{G})$ as a C*-algebra over $\mb{K}$.
The existence of universal complex C*-algebras is well-established, see~\cite{Blackadar_ShapeTheory_1985} or~\cite{Loring_LiftingSolutions_1997}, and the existence of universal real C*-algebras follows by the same argument.
It may be of interest to model theorists that, in both cases, existence of universal C*-algebras is just an application of the continuous form of the classical fact that strict universal Horn theories admit all initial term models.

Note that if $\mc{G}$ is a set of generators and $\mc{R}$ is a set of relations over $\mb{R}$, then $\urep{\mb{C}}{\mc{G}}{\mc{R}}$ is the complexification of $\urep{\mb{R}}{\mc{G}}{\mc{R}}$.

We have not required that the C*-algebras be unital.
If we want to speak of universal C*-algebras among unital C*-algebras, we often need to explicitly require a generator for the identity and relations describing that it is indeed the identity.
To that end, we define here \[\Iden(e;X) = \{e^{2} = e^{*} = e\} \cup \{ex = xe = x \st x \in X\}\] for all indeterminates $e$ and sets of indeterminates $X$.

Let $G$ be a countable discrete group with identity $e$.

\begin{definition}
  If we let $\mc{G} = G$ and let \[\mc{R} = \{x^{*}x = xx^{*} = e \st x \in \mc{G}\} \cup \{xy = z \st x,y,z \in \mc{G} \text{ and } xy = z \text{ in } G\},\] then the universal C*-algebra of $(\mc{G},\mc{R})$ over $\mb{K}$ is called the \inlinedef{universal group C*-algebra} of $G$ over $\mb{K}$ and denoted $\ugcsa{\mb{K}}{G}$.
\end{definition}

There is another important way to form group C*-algebras.
Let $\ell^{2}G$ be the Hilbert space of all square summable functions from $G$ to $\mb{C}$.
Then $G$ acts on $\ell^{2}G$ by left multiplication on the standard orthonormal basis, and in this way can be embedded as a set of operators in $\blo{\mb{K}}{\ell^{2}G}$.
\begin{definition}
The C*-algebra over $\mb{K}$ generated by $G$ in $\blo{\mb{K}}{\ell^{2}G}$ is called the \inlinedef{reduced group C*-algebra} of $G$ over $\mb{K}$ and denoted $\rgcsa{\mb{K}}{G}$.
\end{definition}

The following lifting property will come in handy.
A proof for the complex case can be found in~\cite{RordamEtal_IntroductionKtheory_2000} (2.2.10), and the real case follows from complexification.
\begin{fact}\label{norm lifts}
  Let $A$ and $B$ be C*-algebras over $\mb{K}$.
  If $\pi : A \to B$ is a surjective $*$-homomorphism, then for all $b \in B$ there exist $a \in A$ such that $\norm{a}_{A} = \norm{b}_{B}$ and $\pi(a) = b$.
\end{fact}

When dealing with finite-dimensional C*-algebras, it will be useful to remember they are von Neumann algebras.

\begin{definition}
  Let $\mc{H}$ be a Hilbert space over $\mb{K}$.
  For $S \subseteq \blo{\mb{K}}{\mc{H}}$, we define the \inlinedef{commutant} $S'$ of $S$ in $\blo{\mb{K}}{\mc{H}}$ by \[S' = \{x \in \blo{\mb{K}}{\mc{H}} \st xs= sx \text{ for } s \in S\}.\]
\end{definition}
\begin{definition}[Von Neumann double commutant theorem]\label{vndct}
  Let $\mc{H}$ be a Hilbert space over $\mb{K}$.
  If $A \subseteq \blo{\mb{K}}{\mc{H}}$ is a C*-algebra over $\mb{K}$ which contains the identity, then $A$ is a von Neumann algebra if and only if $A'' = A$ in $\blo{\mb{K}}{\mc{H}}$.
\end{definition}

\subsection{Computability}
We present the basics of effective metric structure theory in the context of C*-algebras.
Our presentation is an instance of the general framework for arbitrary metric structures developed by Franklin and McNicholl in~\cite{FranklinMcnicholl_DegreesLowness_2020}.
See also~\cite{ClaninEtal_AnalyticComputable_2019},\cite{Melnikov_ComputablyIsometric_2013}, or~\cite{Pour-elRichards_ComputabilityAnalysis_2017} for a treatment of Banach spaces.
\begin{definition}
  Given a separable C*-algebra $A$ over $\mb{K}$, a \textbf{presentation} of $A$ is a pair $(A, \vk{a})$, where $\vk{a}$ is a countable sequence of elements of $A$ such that $\vk{a}$ generates $A$ as a C*-algebra over $\mb{K}$.
\end{definition}
Every separable C*-algebra admits a presentation, just consider any countable dense subset.
The presentation is \inlinedef{finitely generated} if the length of $\vk{a}$ is finite.
We refer to the elements of $\vk{a}$ as the \textbf{special points} of the presentation.

We restrict our attention to the class of \textbf{rational polynomials}, where a real polynomial is rational if its coefficients belong to $\mb{Q}$, and a complex polynomial is rational if its coefficients belong to $\mb{Q}(i)$.
If $p$ is a rational *-polynomial in $n$ noncommuting variables with no constant term, and $(A,\vk{a})$ is a presentation of $A$, then we say $p(a_{i_{1}},\ldots,a_{i_{n}})$ is a \textbf{rational point} of $(A,\vk{a})$ for $i_{1},\ldots,i_{n} \in \mb{N}$.

\begin{definition}
  A presentation $A^{\dagger}$ of a C*-algebra $A$ over $\mb{K}$ is \inlinedef{computable} if there is an effective procedure which, when given a rational point $r$ of $A^{\dagger}$ and $k \in \mb{N}$, returns a rational $q \in \mb{Q}$ such that $|\norm{r} - q| < 2^{-k}$.
\end{definition}

Here are some standard computable presentations.
\begin{examples}
\hfill

\begin{enumerate}[label=(\roman*)]
    \item $(\cts{\mb{K}}{[0,1]}, (1,x))$, where $x : [0,1] \to \mb{K}$ is the identity function. The rational points are the rational *-polynomials in $x$, which are dense by Stone-Weierstrass. 
    \item $(M_n(\mb{K}), (e_{ij})_{1 \leq i,j \leq n})$, where each $e_{ij}$ is $1$ in the $(i,j)$ entry and $0$ in all others. The rational points are the rational *-polynomials in $(e_{ij})_{1 \leq i,j \leq n}$, which are clearly dense.
\end{enumerate}
\end{examples}

Let $A$ be a C*-algebra over $\mb{K}$ with computable presentation $A^{\dagger}$.
We say $x \in A$ is a \inlinedef{computable point} of $A^{\dagger}$ if there is an effective procedure which, when given $k \in \mb{N}$, returns a rational point $r$ of $A^{\dagger}$ such that $\norm{x - r} < 2^{-k}$.
Trivially, every rational point of $A^{\dagger}$ is a computable point of $A^{\dagger}$.
A sequence $(x_{n})_{n\in\mb{N}}$ of computable points of $A^{\dagger}$ is \inlinedef{uniformly computable} from $A^{\dagger}$ if there is an effective procedure which, when given $n \in \mb{N}$ and $k \in \mb{N}$, returns a rational point $r$ of $A^{\dagger}$ such that $\norm{x_{n} - r} < 2^{-k}$.

Let $A$ and $B$ be C*-algebras over $\mb{K}$ with computable presentations $A^\dagger$ and $B^\dagger$ respectively.
Let $\varphi : A \to B$ be a *-homomorphism.
Then $\varphi$ is a \inlinedef{computable *-homomorphism} from $A^{\dagger}$ to $B^{\dagger}$ if the images of rational points of $A^\dagger$ are uniformly computable with respect to $B^\dagger$.
Note this notion of computable map agrees with the usual one, as in~\cite{FranklinMcnicholl_DegreesLowness_2020}, since *-homomorphisms are Lipschitz.
If $\varphi$ is bijective and $\varphi$ is computable, then $\varphi^{-1}$ is computable, and we say $\varphi$ is a \textbf{computable isomorphism} from $A^\dagger$ to $B^\dagger$.
A C*-algebra $A$ over $\mb{K}$ is \textbf{computably categorical} if for all computable presentations $A^{\dagger}$ and $A^{+}$ of $A$, there exists a computable isomorphism from $A^{\dagger}$ to $A^{+}$.

We also consider computability properties of closed subsets of C*-algebras as in~\cite{BrattkaPresser_ComputabilitySubsets_2003}.

Let $A$ be a C*-algebra over $\mb{K}$ with computable presentation $A^{\dagger}$.
An open (resp.\ closed) \textbf{rational ball} of $A^{\dagger}$ is an open (resp.\ closed) ball in $A$ whose center is a rational point of $A^{\dagger}$ and whose radius is a positive dyadic rational. We require the radius to be dyadic to better integrate with the framework of continuous first-order logic established in~\cite{YaacovPedersen_ProofCompleteness_2010}.

Let $S$ be a closed subset of $A$. If the set of all open rational balls of $A^{\dagger}$ that intersect $S$ is c.e., then $S$ is \textbf{c.e.\ closed}.
If there is a c.e.\ set of open rational balls of $A^{\dagger}$ whose union is the complement of $S$, then $S$ is \textbf{co-c.e.\ closed}.
Together, if $S$ is c.e.\ closed and co-c.e.\ closed, then $S$ is \textbf{computable closed}.
If the set of all closed rational balls of $A^{\dagger}$ that do not intersect $S$ is c.e., then $S$ is \textbf{strongly co-c.e.\ closed}.
Similarly, if $S$ is c.e.\ closed and strongly co-c.e.\ closed, then $S$ is \textbf{strongly computable closed}.

Here are some fundamental propositions in the computable structure theory for C*-algebras. The straightforward proofs are left to the reader.

\begin{proposition}\label{moc}
  There is an effective procedure which, when given a rational noncommutative *-polynomial over $\mb{K}$ with no constant term $q(z_1,\ldots,z_n)$, a rational bound $M$, and $k \in \mb{N}$, returns $j \in \mb{N}$ such that, for all Banach *-algebras $B$ and all $v_1,\ldots,v_n,w_1,\ldots,w_n \in B$, if $\max_i\norm{v_i} \leq M$, $\max_i\norm{w_i} \leq M$, and $\max_i\norm{v_i - w_i} \leq 2^{-j}$, then \[\norm{q(v_1,\ldots,v_n) - q(w_1,\ldots,w_n)} < 2^{-k}.\]
\end{proposition}

\begin{proposition}\label{lem: unif comp substructure}
  Let $A$ be a C*-algebra over $\mb{K}$ and let $A^\dagger$ be a computable presentation of $A$.
  Let $\vk{x}$ be a sequence of uniformly computable points of $A^\dagger$.
  Let $B = \cgen{\mb{K}}{\vk{x}} \subseteq A$, the C*-algebra over $\mb{K}$ generated by $\vk{x}$.
  Then $(B, \vk{x})$ is a computable presentation of $B$ and the inclusion map from $(B, \vk{x})$ to $A^\dagger$ is computable.
\end{proposition}




\begin{corollary}\label{cor: isomorphic to finite subset}
  Let $A$ be a C*-algebra over $\mb{K}$ and let $A^\dagger$ be a computable presentation of $A$.
  Let $\vk{x}$ be a uniformly computable sequence of points of $A^\dagger$ such that $\cgen{\mb{K}}{\vk{x}} = A$.
  Then $(A, \vk{x})$ is a computable presentation of $A$ which is computably isomorphic to $A^\dagger$ via the identity map.
\end{corollary}



\subsection{Linear algebra}

When working with real C*-algebras, we will need to do linear algebra over $\mb{H}$.
Let $\mb{D}$ be one of $\mb{R}$, $\mb{H}$, or $\mb{C}$, viewed as a C*-algebra over corresponding $\mb{K}$.
We develop basic linear algebra facts over $\mb{D}$ for that purpose. See~\cite{Cohn_Algebra_1995} or~\cite{Lorenz_Algebra_2008} for the fundamentals of linear algebra, and~\cite{FarenickPidkowich_SpectralTheorem_2003} for an exploration of linear algebra for quaternions.

A \inlinedef{vector space} $V$ over $\mb{D}$ is a right $\mb{D}$-module.
Any vector space over $\mb{D}$ can be viewed as a vector space over $\mb{K}$ by restricting scalars to $\mb{K}1$ inside $\mb{D}$.
An \inlinedef{inner product space} over $\mb{D}$ is a vector space $V$ over $\mb{D}$ equipped with an inner product $\langle\;,\;\rangle : V \times V \to \mb{D}$ that has the following properties for $x,y,z \in V$ and $a \in \mb{D}$:
\begin{enumerate}[label= (\roman*)]
  \item $\langle x + y, z \rangle = \langle x, z \rangle + \langle y,z \rangle$,
  \item $\langle x, y + z \rangle = \langle x,y \rangle + \langle x,z \rangle$,
  \item $\langle xa,y \rangle = a^{*}\langle x,y \rangle$,
  \item $\langle x, ya \rangle = \langle x,y \rangle a$,
  \item $\langle x,y \rangle^{*} = \langle y, x \rangle$,
        \item $\langle x,x \rangle \geq 0$ and if $\langle x,x \rangle = 0$, then $x = 0$.
\end{enumerate}
The inner product determines a norm on $V$ given by $x \mapsto \sqrt{\langle x,x \rangle}$, where we identify the self-adjoint elements of $\mb{D}$ with $\mb{R}$.
A \inlinedef{Hilbert space} $\mc{H}$ over $\mb{D}$ is an inner product space over $\mb{D}$ which is complete with respect to the induced norm.
Note $\mc{H}$ can also be viewed as a Hilbert space over $\mb{R}$ when equipped with the inner product $(x,y) \mapsto \frac{1}{2}(\langle x,y \rangle + \langle x,y\rangle^{*})$.

We view $\mb{D}$ as a vector space over itself.
The set of $n$ by $n$ matrices $M_{n}(\mb{D})$ with entries in $\mb{D}$ can be identified with the set of $\mb{D}$-linear operators $\blo{\mb{D}}{\mb{D}^{n}}$ on $\mb{D}^{n}$.
Rank is well-defined since $\mb{D}$ is a division ring so has the invariant basis number property.
Any $\mb{D}$-linear operator on a Hilbert space $\mc{H}$ over $\mb{D}$ is also $\mb{K}$-linear, so we can view $\blo{\mb{D}}{\mc{H}}$ as a subspace of $\blo{\mb{K}}{\mc{H}}$.
For all $a \in \mb{D}$, define $R_a : \mc{H} \to \mc{H}$ to be multiplication on the right by $a$.
Then the commutant $\blo{\mb{D}}{\mc{H}}'$ of $\blo{\mb{D}}{\mc{H}}$ in $\blo{\mb{K}}{\mc{H}}$ is $\{R_{a} \st a \in \mb{D}\}$.
In particular, the center of $\blo{\mb{D}}{\mc{H}}$ can be identified with the center of $\mb{D}$.

Of course, $\mb{D}$ itself is a Hilbert space over $\mb{K}$ when equipped with the natural inner product.
We identify a standard orthonormal basis $Z$ of $\mb{D}$ over $\mb{K}$ as follows.
We let
\begin{itemize}
        \item $Z = \{1\}$ if $\mb{D} = \mb{K} = \mb{R}$ or $\mb{D} = \mb{K} = \mb{C}$,
  \item $Z = \{1, i\}$ if $\mb{D} = \mb{C}$ and $\mb{K} = \mb{R}$,
  \item $Z = \{1, i, j, k\}$ if $\mb{D} = \mb{H}$ and $\mb{K} = \mb{R}$.
\end{itemize}

We also define a system of matrix units which characterize $M_{n}(\mb{D})$ as a C*-algebra over $\mb{K}$ up to isomorphism.
For any element $t$ in $Z \cup \{-z \st z \in Z\}$, we define $\epsilon(t) \in \{1,-1\}$ and $b(t) \in Z$ such that $t = \epsilon(t)b(t)$.
We say $(f_{rs}^{z} \st 1 \leq r,s \leq n, z \in Z)$ is a \inlinedef{system of matrix units} for $M_{n}(\mb{D})$ over $\mb{K}$ if $(f_{rs}^{z})^{*} = \epsilon(z^{*})f_{sr}^{z}$ and $f_{rs}^{z}f_{uv}^{w} = \epsilon({zw})\delta_{su}f^{b(zw)}_{rv}$ for $1 \leq r,s,u,v \leq n$ and $z,w \in Z$.
The standard system of matrix units is then just $(e_{rs}^{z} \st 1 \leq r,s \leq n, z \in Z)$ where $e_{rs}^{z}$ is the matrix with $z$ in the $(r,s)$ entry and $0$ in all other entries for $1 \leq r,s \leq n$ and $z \in Z$.
When the superscript of a matrix unit is $1$, we will often suppress it.

We will use Skolem-Noether to characterize the $*$-automorphisms of $M_{n}(\mb{D})$.
\begin{fact}[Skolem-Noether]
  Let $k$ be a field.
  Let $A$ be a finite-dimensional $k$-algebra which is simple and has center $k$.
  Then every automorphism of $A$ is of the form $x \mapsto v^{-1}xv$ for some unit $v \in A$.
\end{fact}
\begin{corollary}\label{SN_cor}
  If the center of $\mb{D}$ is exactly $\mb{K}$, then every $*$-automorphism of $M_{n}(\mb{D})$ is of the form $x \mapsto u^{*}xu$ for some unitary $u \in M_{n}(\mb{D})$.
\end{corollary}
\begin{proof}
  Let $\varphi$ be a $*$-automorphism of $M_{n}(\mb{D})$.
  It can be observed that $M_{n}(\mb{D})$ is a simple finite-dimensional $\mb{K}$-algebra with center $\mb{K}$, so, by the previous fact, there exists an invertible $v \in M_{n}(\mb{D})$ such that $\varphi(x) = v^{-1}xv$ for $x \in M_{n}(\mb{D})$.
  For $x \in M_{n}(\mb{D})$,
  \[xvv^{*} = v(v^{-1}xv)v^{*} = v\varphi(x)v^{*} = v\varphi(x^{*})^{*}v^{*} = v(v^{*}x(v^{-1})^{*})v^{*} = vv^{*}x.\]
  Hence $vv^{*}$ belongs to the center of $M_{n}(\mb{D})$, so is of the form $R_{a}$ for some $a \in \mb{K}$.
  In fact, since $vv^{*}$ is self-adjoint and nonzero, it must be that $a > 0$.
  Let $b = a^{-1/2}$ and let $u = R_{b} v$.
  Then $uu^{*} = R_{b}vv^{*}R_{b} = R_{b}R_{a}R_{b} = I$.
  By uniqueness of the inverse, $u$ is unitary.
  Furthermore, for $x \in M_{n}(\mb{D})$, $u^*xu = v^{-1}R_{1/b}xR_{b}v = v^{-1}xv = \varphi(x)$.
\end{proof}

\section{C.e.\ presentations and word problems}\label{sec3}

Here we introduce c.e.\ presentations for C*-algebras, borrowing the terminology from~\cite{Melnikov_ComputableAbelian_2014}, and explore their similarity to the group situation.

We consider universal C*-algebras where the set of generators forms a sequence.
Let $\urep{\mb{K}}{\vk{x}}{\mc{R}}$ be a universal C*-algebra over $\mb{K}$, where we identify the elements of $\vk{x}$ with their image under the associated representation.
The standard presentation of $\urep{\mb{K}}{\vk{x}}{\mc{R}}$ is then $(\urep{\mb{K}}{\vk{x}}{\mc{R}}, \vk{x})$, which we simply denote by $\urep{\mb{K}}{\vk{x}}{\mc{R}}$.

On a set of generators, a relation $\norm{p(x_{1},\ldots,x_{n})} \leq a$ is \inlinedef{rational} if $p$ is rational and $a$ is a positive dyadic rational.
\begin{definition}
  A presentation $A^{\dagger}$ of a C*-algebra $A$ over $\mb{K}$ is \inlinedef{c.e.}\ if $A^{\dagger}$ is the standard presentation $\urep{\mb{K}}{\vk{x}}{\mc{R}}$ for some c.e.\ set of rational relations $\mc{R}$.
\end{definition}
If $\vk{x}$ and $\mc{R}$ are finite, we say $A^{\dagger}$ is \inlinedef{finitely c.e.}

There is another c.e.\ notion for presentations that one might consider in any presented metric structure.
We give the definition for C*-algebras, see~\cite{BazhenovEtal_ComputableStone_2021} for the definition for Polish metric spaces.
\begin{definition}
  A presentation $A^{\dagger}$ of a C*-algebra $A$ over $\mb{K}$ is \inlinedef{right-c.e.}\ if there is an effective procedure which, when given a rational point $r$ of $A^{\dagger}$, enumerates a decreasing sequence $(q_{n})_{n \in \mb{N}}$ of rationals such that $\lim_{n \to \infty} q_{n} = \norm{r}$.
\end{definition}

In the following theorem, we show the two notions agree. We use the framework of continuous first-order logic, see~\cite{YaacovPedersen_ProofCompleteness_2010} for an overview of the underlying logic and~\cite{FarahEtal_ModelTheory_2018} for a reference on its applications to complex C*-algebras.
\begin{theorem}\label{c.e. norm}
  Let $A$ be a C*-algebra over $\mb{K}$ with presentation $A^{\dagger}$.
  Then $A^{\dagger}$ is c.e.\ if and only if it is right-c.e.
  Furthermore, if $A^{\dagger} = \urep{\mb{K}}{\vk{x}}{\mc{R}}$ for some c.e.\ set of rational relations $\mc{R}$ in $\vk{x}$, we can effectively determine a procedure which witnesses that $A^{\dagger}$ is right-c.e.\ from a computable enumeration of $\mc{R}$.
\end{theorem}
\begin{proof}
  Suppose $A^{\dagger}$ is c.e., so $A^{\dagger}$ is the standard presentation $\urep{\mb{K}}{\vk{x}}{\mc{R}}$ for some c.e.\ set of rational relations $\mc{R}$ in $\vk{x}$.

  Let $T$ be the continuous first-order theory of C*-algebras over $\mb{K}$.
  We expand our language to include additional constants for the generators $\vk{x}$.
  Note if $x_{i}$ is a member of $\vk{x}$, then there is a relation of the form $\norm{x_{i}} \leq M$ in $\mc{R}$, so we can put $x_{i}$ in an appropriate sort.
  Given a rational point $r(\vk{x})$ of $A^{\dagger}$, we enumerate through all formal deductions from $T \cup \mc{R}$.
  For each formal deduction, we output a positive dyadic rational $q_{n}$ if $q_{n} \leq q_{i}$ for all $i < n$ and the formal deduction witnesses that $T \cup \mc{R} \vdash \norm{r(\vk{x})} \leq q_{n}$.
  Note this procedure is defined uniformly in the computable enumeration of $\mc{R}$.
  By Pavelka-style completeness, $\norm{r(\vk{x})}$ is the infimum of all positive dyadic rationals $d$ such that $T \cup \mc{R} \vdash \norm{r(\vk{x})} \leq d$.
  Thus, we have enumerated a decreasing sequence of rationals $(q_{n})_{n \in \mb{N}}$ such that $\lim_{n \to \infty} q_{n} = \norm{r(\vk{x})}$.

  Conversely, suppose $A^{\dagger}$ is right-c.e.
  Then there is an effective procedure which, when given a rational point $r$ of $A^{\dagger}$, enumerates a decreasing sequence $(q_{n})_{n \in \mb{N}}$ of dyadic rationals such that $\lim_{n \to \infty} q_{n} = \norm{r}$.
  Let $\vk{a}$ be such that $A^{\dagger} = (A,\vk{a})$, and let $\vk{x}$ be a sequence of generators of the same length.
  Let $\mc{R}$ be the set of relations $\norm{r(\vk{x})} \leq d$ where $r(\vk{a})$ is a rational point of $A^{\dagger}$ and $d$ is one of the positive dyadic rationals enumerated by the procedure given $r(\vk{a})$.
  Then $\mc{R}$ is c.e., and $A^{\dagger} = \urep{\mb{K}}{\vk{x}}{\mc{R}}$.
\end{proof}

It may interest some that the statement and proof directly generalize from C*-algebras to metric models of a c.e.\ strict universal Horn theory.

Essentially, with~\autoref{c.e. norm}, we have rephrased c.e.\ in terms of the norm so that it mirrors the definition of computable.
From this, it becomes clear that if $A^{\dagger}$ is a computable presentation of a C*-algebra $A$ over $\mb{K}$, then $A^{\dagger}$ is also c.e.

If a finitely presented group is residually finite, then it has solvable word problem, as shown by Dyson in~\cite{Dyson_WordProblem_1964}.
We want to establish an analogous result for C*-algebras.
\begin{definition}
  A C*-algebra over $\mb{K}$ is called \inlinedef{RFD} if its finite-dimensional representations form a separating family.
\end{definition}

In~\cite{FritzEtal_CanYou_2014}, the authors showed that the standard presentation for a universal group C*-algebra of a finitely presented RFD group is computable, and noted that their result generalizes to arbitrary finitely-presented *-algebras.
Their proof works for finitely c.e.\ presentations of C*-algebras over $\mb{K}$ with only a small modification.
We include the proof here to emphasize that their use of semidefinite programming is not required, and that the argument is uniform in the sense we describe.
\begin{theorem}
  If $A$ is an RFD C*-algebra over $\mb{K}$, then any finitely c.e.\ presentation $A^{\dagger} = \urep{\mb{K}}{\vk{x}}{\mc{R}}$ of $A$ is computable.
  Furthermore, from $\mc{R}$ one can effectively determine a procedure which witnesses that the presentation is computable.
\end{theorem}
\begin{proof}
  We write $\norm{\cdot}_f$ for the norm on the direct sum of all finite-dimensional representations of $A$.
  Since $A$ is RFD, we have $\norm{\cdot}_{A} = \norm{\cdot}_f$.

  For any $n \in \mb{N}$, consider the set $X_n = \{\vk{m} \in M_n(\mb{K})^{|\vk{x}|} \st \mc{R}[\vk{m}] \text{ holds in } M_{n}(\mb{K})\}$ as a subset of $\mb{R}^{k|\vk{x}|n^2}$, where $k$ is the dimension of $\mb{K}$ over $\mb{R}$.
  Note $X_n$ is compact since every relation in $\mc{R}$ is a closed condition and we require for every generator $x_i$ that there is a relation of the form $\norm{x_i} \leq M_i$ in $\mc{R}$.
  Furthermore, $X_n$ is definable in the language of real closed fields since $\mc{R}$ is finite and if $\norm{p(\vk{x})} \leq d$  is a relation in $\mc{R}$, then $\norm{p(\vk{m})} \leq d$ holds in $M_{n}(\mb{K})$ if and only if \[(\forall \lambda \in \mb{R}) [\textstyle\det_n\big(p(\vk{m})^*p(\vk{m}) - \lambda I\big) = 0 \to \lambda \leq d^2].\]
  Let $D_n = \{\xi \in \mb{K}^n : \norm{\xi} \leq 1\}$.
  Fix a rational point $q(\vk{x}) \in A^{\dagger}$.
  Define $F_{n,q} : X_n \times D_n \to \mb{R}$ by $F_{n,q}(\vk{m},\xi) = \norm{q(\vk{m})\xi}^2$, and let $\alpha_{n,q}$ be the maximum value of $F_{n,q}$.
  Following closely, we see there is actually an effective procedure which, when given inputs $n$ and $q$, returns a formula that defines $\alpha_{n,q}$ in the language of real closed fields.
  Applying the effective quantifier elimination given by Tarski-Seidenberg, we see this formula can even be made quantifier-free, so $\alpha_{n,q}$ is computable uniformly in $n$ and $q$.

    Observe $\norm{q(\vk{x})}_A = \norm{q(\vk{x})}_f \leq \sup\{\alpha_{n,q}^{1/2} \st n \in \mb{N}\}$ since every finite-dimensional representation of $A$ can be embedded in $M_n(\mb{K})$ for sufficiently large $n$.
    By the universality of $\urep{\mb{K}}{\vk{x}}{\mc{R}}$, we also have $\norm{q(\vk{x})}_A \geq \sup\{\alpha_{n,q}^{1/2} \st n \in \mb{N}\}$.
    Hence $\norm{q(\vk{x})}_A = \sup\{\alpha_{n,q}^{1/2} : n \in \mb{N}\}$, so along with \autoref{c.e. norm} we can conclude that $\norm{q(\vk{x})}$ is computable uniformly in $q$.
    Thus $A^\dagger$ is computable.
  This procedure is uniform in $\mc{R}$ since the formula which defines $\alpha_{n,q}$ can be effectively determined uniformly in $\mc{R}$, and the procedure from \autoref{c.e. norm} is uniform in $\mc{R}$.
\end{proof}

We would like to study word problems associated to presentations of C*-algebras.
Although the category of C*-algebras over $\mb{K}$ does not admit free objects, we can recover a lot of their utility for word problems by considering universal contraction algebras.

\begin{definition}
  For $n \in \mb{N}$, we let $\contr{\mb{K}}{n}$ denote $\urep{\mb{K}}{c_1,\ldots,c_n}{\norm{c_j} \leq 1}$, the \inlinedef{universal contraction C*-algebra} over $\mb{K}$ on $n$ generators.
  Similarly, we let $\contr{\mb{K}}{\omega}$ denote $\urep{\mb{K}}{\{c_{j} \st j \in \mb{N}\}}{\norm{c_{j}} \leq 1}$, the universal contraction C*-algebra over $\mb{K}$ on infinitely many generators.
\end{definition}
To avoid confusion, we will always use $\vk{c}$ to refer to the generators of a universal contraction C*-algebra.

In order to study the computability properties of subsets of the standard presentation $\contr{\mb{K}}{n}$, we first need to establish that the presentation is computable.
The following is a standard fact, see~\cite{Loring_LiftingSolutions_1997} for details.
\begin{fact}
  $\contr{\mb{K}}{n}$ is RFD for every $n \in \mb{N} \cup \{\omega\}$.
\end{fact}

Together with the previous theorem, we have the following.
\begin{corollary}
 The standard presentation $\contr{\mb{K}}{n}$ is computable for every $n \in \mb{N}$.
\end{corollary}

Not only that, but the effective procedure which witnesses that $\contr{\mb{K}}{n}$ is computable is uniform in $n$.
If $p(c_{1},\ldots,c_{k})$ is a rational point of $\contr{\mb{K}}{\omega}$, then $\norm{p(c_{1},\ldots,c_{k})}_{\contr{\mb{K}}{\omega}} = \norm{p(c_{1},\ldots,c_{k})}_{\contr{\mb{K}}{k}}$, so we have the following.

\begin{corollary}
  The standard presentation $\contr{\mb{K}}{\omega}$ is computable.
\end{corollary}

Now we are in a position to define the word problem.
When considering the word problem, it is convenient if we assume every presentation $(A,\vk{a})$ satisfies $\norm{a_j} \leq 1$ for all $j$.
\begin{definition}
  Let $A$ be a C*-algebra over $\mb{K}$ with presentation $(A,\vk{a})$.
  The \inlinedef{word problem} of $(A,\vk{a})$ is the kernel of the natural quotient map from $\contr{\mb{K}}{|\vk{a}|}$ onto $A$.
\end{definition}

As with groups, there are relationships between the word problem and the presentation.
\begin{theorem}\label{c.e. wp}
  Let $A$ be a C*-algebra over $\mb{K}$ with presentation $A^{\dagger}$.
  The following are equivalent:
  \begin{enumerate}[label= (\alph*)]
    \item $A^{\dagger}$ is the standard presentation $\urep{\mb{K}}{\vk{x}}{\mc{S}}$ where $\mc{S}$ is a computable set of relations,
    \item the word problem of $A^{\dagger}$ is c.e.\ closed,
    \item $A^\dagger$ is a c.e.\ presentation.
  \end{enumerate}
\end{theorem}
\begin{proof}
  (a) $\implies$ (c) clearly holds.

  (c) $\implies$ (a). This follows by the standard argument known as Craig's trick.
  Since $A^{\dagger}$ is c.e., $A^{\dagger}$ is the standard presentation $\urep{\mb{K}}{\vk{x}}{\mc{R}}$ for some c.e.\ set of rational relations $\mc{R}$.
  For every relation $\norm{p(\vk{x})} \leq d$ in $\mc{R}$, include the relation $\norm{p(\vk{x}) + k0} \leq d$ in $\mc{S}$, where $k$ encodes a Turing computation which witnesses that $\norm{p(\vk{x})} \leq d$ belongs to $\mc{R}$.
  Then $\mc{S}$ is computable and $A^{\dagger}$ is the standard presentation $\urep{\mb{K}}{\vk{x}}{\mc{S}}$.

  Let $\vk{a}$ be such that $A^{\dagger} = (A,\vk{a})$, and let $\vk{c}$ be the corresponding sequence of generators for $\contr{\mb{K}}{|\vk{a}|}$.
  Let $N$ be the word problem of $A^{\dagger}$.

  (b) $\implies$ (c).
  Given a rational point $r(\vk{a})$ of $A^{\dagger}$, we enumerate the set of all open rational balls of $\contr{\mb{K}}{|\vk{a}|}$ that intersect $N$.
  For each rational ball, we output a positive dyadic rational $q_{n}$ if $q_{n} \leq q_{i}$ for all $i < n$, and the open rational ball is of the form $B(r(\vk{c}), q_{n})$.
  Note $B(r(\vk{c}),q_{n})$ intersects $N$ if and only if $\norm{r(\vk{a})} < q_{n}$.
  Then $(q_{n})_{n \in \mb{N}}$ is a decreasing sequence of rationals such that $\lim_{n \to \infty} q_{n} = \norm{r(\vk{a})}$.
  By~\autoref{c.e. norm}, we conclude the presentation $A^{\dagger}$ is c.e.

  (c) $\implies$ (b).
  We enumerate through all rational points of $\contr{\mb{K}}{|\vk{a}|}$ and all positive dyadic rationals.
  For each rational point $r(\vk{c})$ and positive dyadic rational $d$, we begin an enumeration of the decreasing sequence $(q_{n})_{n \in \mb{N}}$ of rationals determined by \autoref{c.e. norm} on input $r(\vk{a})$, and output $B(r(\vk{c}),d)$ if ever $q_{n} < d$ for some $n \in \mb{N}$.
  Since $B(r(\vk{c}),d)$ intersects $N$ if and only if $\norm{r(\vk{a})} < d$, we have shown $N$ is c.e.\ closed.
\end{proof}

We can even define when a presentation is computable in terms of the word problem.
\begin{theorem}\label{computable wp}
  Let $A$ be a C*-algebra over $\mb{K}$ with presentation $A^\dagger$.
  Then $A^{\dagger}$ is computable if and only if the word problem of $A^{\dagger}$ is strongly computable closed.
\end{theorem}
\begin{proof}
  Let $\vk{a}$ be such that $A^{\dagger} = (A,\vk{a})$, and let $\vk{c}$ be the corresponding sequence of generators for $\contr{\mb{K}}{|\vk{a}|}$.
  Let $N$ be the word problem of $A^{\dagger}$.

  Suppose $A^{\dagger}$ is a computable presentation.
  Then $A^{\dagger}$ is a c.e.\ presentation, so $N$ is c.e.\ closed by \autoref{c.e. wp}.
  We show $N$ is strongly co-c.e.\ closed.
  We enumerate through all rational points $r(\vk{a})$ of $A^{\dagger}$, positive dyadic rationals $d$, and $k \in \mb{N}$.
  The computable presentation $A^{\dagger}$ determines a rational $q$ such that $|\norm{r(\vk{a})} - q| < 2^{-k}$.
  If $d \leq q - 2^{-k}$, then we output the closed rational ball $\widehat{B}(r(\vk{c}), d)$.
  If $d$ is a positive dyadic rational such that $\widehat{B}(r(\vk{c}), d) \cap N = \emptyset$, then $\norm{r(\vk{a})} > d$ by \autoref{norm lifts}.
  So, if $k \in \mb{N}$ is such that $2^{-k+1} < \norm{r(\vk{a})} - d$, then $d \leq q - 2^{-k}$ for any rational $q$ such that $|\norm{r(\vk{a})} - q| < 2^{-k}$.
  Hence $\widehat{B}(r(\vk{c}), d)$ is eventually output.

  Conversely, suppose $N$ is strongly computable closed.
  If $B(r(\vk{c}),d)$ is an open rational ball, then it intersects $N$ if and only if $\norm{r(\vk{a})} < d$.
  Similarly, if $\widehat{B}(r(\vk{c}),e)$ is a closed rational ball, then $\widehat{B}(r(\vk{c}),e) \cap N = \emptyset$ if and only if $\norm{r(\vk{a})} > e$ by \autoref{norm lifts}.
  We define an effective procedure which witnesses that $A^{\dagger}$ is computable.
  Given a rational point $r(\vk{a}) \in A^{\dagger}$ and $k \in \mb{N}$, we begin enumerating all open rational balls $B(r(\vk{c}), d_{i})$, centered at $r(\vk{c})$, that intersect $N$.
  We also begin enumerating all closed rational balls $\widehat{B}(r(\vk{c}),e_{j})$, centered at $r(\vk{c})$, such that $\widehat{B}(r(\vk{c}),e_{j}) \cap N = \emptyset$.
  If ever $d_{i} - e_{j} < 2^{-k+1}$ for some $i,j \in \mb{N}$, we return $\frac{1}{2}(d_{i} - e_{j})$.
\end{proof}

Consequently, if $A^{\dagger}$ is a computable presentation of a C*-algebra $A$ over $\mb{K}$, then the word problem of $A^{\dagger}$ is computable closed.
In the group situation, we know the converse holds.

\begin{question}
  Is there a presentation $A^{\dagger}$ of a C*-algebra $A$ over $\mb{K}$ such that $A^{\dagger}$ is not computable but has computable closed word problem?
\end{question}

In~\cite{Kuznetsov_AlgorithmsOperations_1958}, Kuzntsov proved that a recursively presented simple group has solvable word problem. Analogously, we have the following theorem.
\begin{theorem}\label{thm: simple computable rup implies computable}
  If $A$ is a simple C*-algebra over $\mb{K}$, then any c.e.\ presentation $A^{\dagger}$ of $A$ is computable.
\end{theorem}
\begin{proof}
  Let $\mc{R}$ be a c.e.\ set of rational relations such that $A^{\dagger}$ is $\urep{\mb{K}}{\vk{x}}{\mc{R}}$.

  Let $T$ be the continuous first-order theory of C*-algebras over $\mb{K}$.
  We expand our language to include additional constants for the generators $\vk{x}$.
  Fix a rational point $q(\vk{x})$ of $A^{\dagger}$ and a positive dyadic rational $\ell$ such that $\norm{q(\vk{x})} > \ell$.
  Since $A$ is simple, if $r(\vk{x})$ is a rational point of $A^{\dagger}$ and $d$ is a positive dyadic rational, then $\urep{\mb{K}}{\vk{x}}{\mc{R} \cup \{\norm{r(\vk{x})} \leq d\}}$ is just $A$ if $\norm{r(\vk{x})} \leq d$ and $\{0\}$ if $\norm{r(\vk{x})} > d$.
  Hence by Pavelka-style completeness, $\norm{r(\vk{x})} > d$ if and only if $T \cup \mc{R} \cup \{\norm{r(\vk{x})} \leq d\} \vdash \norm{q(\vk{x})} \leq \ell$.

  We define an effective procedure which witnesses that $A^{\dagger}$ is computable.
  Given a rational point $r(\vk{x}) \in A^\dagger$ and $k \in \mb{N}$, we apply \autoref{c.e. norm} and enumerate a decreasing sequence $(q_n)_{n \in \mb{N}}$ of positive dyadic rationals such that $\lim_{n \to \infty} q_n = \norm{r(\vk{x})}$.
  Let $a \dotminus b$ be $a - b$ if $a \geq b$ and $0$ otherwise for $a,b \in \mb{R}$.
  For each $q_{n}$, we begin enumerating through all formal deductions from $T \cup \mc{R} \cup \{\norm{r(\vk{x})} \leq q_{n} \dotminus 2^{-k}\}$, and we return $q_{n}$ if the formal deduction witnesses that $T \cup \mc{R} \cup \{\norm{r(\vk{x})} \leq q_{n} \dotminus 2^{-k}\} \vdash \norm{q(\vk{x})} \leq \ell$.
\end{proof}

The computability of several standard presentations follows as a direct consequence.

\begin{definition}
  For $2 \leq n < \infty$, the \inlinedef{Cuntz algebra} $\mc{O}(n;\mb{K})$ over $\mb{K}$ is the universal C*-algebra given by \[\urep{\mb{K}}{s_{1},\ldots,s_{n},1}{\Iden(1;s_{1},\ldots,s_{n}) \cup \mc{R}},\]
  where \[\mc{R} = \{s_i^*s_j = \delta_{ij}1 \st i,j \leq n\}\cup\{s_{1}s_{1}^* + \cdots + s_{n}s_{n}^{*}= 1\}.\]
  We also consider $n = \infty$ and define $\mc{O}(\infty;\mb{K})$ to be the universal C*-algebra given by \[\urep{\mb{K}}{\{s_{i} \st i \in \mb{N}\}, 1}{\Iden(1;\{s_{i} \st i \in \mb{N}\}), s_{i}^{*}s_{j} = \delta_{ij}1 \;(i,j \in \mb{N})}.\]
 \end{definition}
 It is known that $\mc{O}(n;\mb{C})$ is simple, see~\cite{Cuntz_SimpleAlgebras_1977} for details.
 Note then $\mc{O}(n;\mb{R})$ is simple by complexification.

\begin{corollary}
  The standard presentation of $\mc{O}(n;\mb{K})$ is computable for $2 \leq n \leq \infty$.
\end{corollary}

\begin{definition}
  For irrational $\theta \in (0,1)$, the \inlinedef{irrational rotation algebra} $A_{\theta}$ is the complex universal C*-algebra given by \[\urep{\mb{C}}{u,v,1}{\Iden(1;u,v) \cup \mc{R}},\]
  where \[\mc{R} = \{u^*u = uu^* = v^*v = vv^* = 1\} \cup \{uv = e^{2i\pi \theta}vu\}.\]
\end{definition}
It is known that $A_{\theta}$ is simple, see~\cite{Davidson_AlgebrasExample_1996} for details.

\begin{corollary}
  Let $\theta \in (0,1)$ be irrational.
  The following are equivalent:
  \begin{enumerate}[label= (\alph*)]
    \item $\theta$ is computable,
    \item the standard presentation $A_\theta$ is c.e.,
    \item the standard presentation $A_\theta$ is computable.
  \end{enumerate}
\end{corollary}
\begin{proof}
    (a) $\implies$ (b).
    Since $\theta$ is computable, $\cos(2\pi \theta)$ and $\sin(2\pi \theta)$ are also computable.
    Let $(a_k)_{k \in \mb{N}}$ be a computable enumeration of rationals such that $|\cos(2\pi\theta) - a_k| \leq 2^{-k}$ for $k \in \mb{N}$, and let $(b_k)_{k \in \mb{N}}$ be a computable enumeration of rationals such that $|\sin(2\pi\theta) - b_k| \leq 2^{-k}$ for $k \in \mb{N}$.
    Let
    \begin{align*}
      \mc{S} = \Iden(1;u,v) &\cup \{u^*u = uu^* = v^*v = vv^* = 1\} \\
      &\cup \{\norm{uv - (a_k +ib_k)vu} \leq 2^{-k+1} : k \in \mb{N}\}.
    \end{align*}
    Then $\mc{S}$ is c.e.\ and $A_{\theta}$ is the standard presentation $\urep{\mb{K}}{u,v,1}{\mc{S}}$.

    (b) $\implies$ (c).
    Since $A_\theta$ is simple, by \autoref{thm: simple computable rup implies computable}, $A_\theta$ is computable.

    (c) $\implies$ (a).
    Note that $uvu^*v^* - vuv^*u^* = 2i\sin(2\pi\theta)1$ is a rational point of $A_\theta$, so $\sin(2\pi\theta)$ is computable.
    Similarly, $uvu^*v^* + vuv^*u^* = 2i\cos(2\pi\theta)1$ is a rational point of $A_\theta$, so $\cos(2\pi\theta)$ is computable.
    The angle $\theta$ can be calculated from $\sin(2\pi \theta)$ and $\cos(2\pi\theta)$ with use of the arctangent function.
    Thus $\theta$ is computable.
  \end{proof}

We now investigate the connections between a group and its group C*-algebras.
We cover some of the same ground as in~\cite{FritzEtal_CanYou_2014}, but our perspective has the advantage of avoiding the use of semidefinite programming.

When working with arbitrary countable groups, we adopt the same language of presentations that we have used for C*-algebras.
We can agin use the framework for arbitrary metric structrues developed in~\cite{FranklinMcnicholl_DegreesLowness_2020} if we view discrete groups as metric structures equipped with the discrete metric.
\begin{definition}
Given a countable discrete group $G$ with identity $e$, we say $(G,\vk{g})$ is a \inlinedef{presentation} of $G$ if $\vk{g}$ is a countable sequence of elements from $G$ that generates $G$ as a group.
The presentation $(G,\vk{g})$ is \inlinedef{computable} if the set of all words $w$, on generators $\vk{g}$, for which $w = e$ in $G$ is computable.

\end{definition}
We say a presentation of $G$ is \inlinedef{c.e.}\ if the presentation witnesses that $G$ is recursively presented, and is \inlinedef{finitely c.e.}\ if the presentation witnesses that $G$ is finitely presented.

Let $G$ be a countable discrete group with presentation $G^{\dagger}$.
We denote by $\ugcsa{\mb{K}}{G^{\dagger}}$ the induced presentation of $\ugcsa{\mb{K}}{G}$, and by $\rgcsa{\mb{K}}{G^{\dagger}}$ the induced presentation of $\rgcsa{\mb{K}}{G}$.
From the definition of $\ugcsa{\mb{K}}{G}$, it is clear that if $G^{\dagger}$ is c.e., then $\ugcsa{\mb{K}}{G^{\dagger}}$ is c.e.
Furthermore, if $G^{\dagger}$ is finitely c.e., then $\ugcsa{\mb{K}}{G^{\dagger}}$ is finitely c.e.


\begin{theorem}
  Let $G$ be a countable discrete group with presentation $G^{\dagger}$.
  Let $N_{\mathrm{uni}}$ and $N_{\mathrm{red}}$ be the word problems of $\ugcsa{\mb{K}}{G^{\dagger}}$ and $\rgcsa{\mb{K}}{G^{\dagger}}$, respectively.
  The following are equivalent:
  \begin{enumerate}[label= (\alph*)]
    \item $G^{\dagger}$ is a computable presentation,
    \item $N_{\mathrm{uni}}$ is c.e.\ closed and $N_{\mathrm{red}}$ is strongly co-c.e.\ closed,
    \item $N_{\mathrm{uni}}$ is c.e.\ closed and the set of rational points which belong to the complement $N_{\mathrm{uni}}^{c}$ is c.e.,
    \item $N_{\mathrm{red}}$ is co-c.e.\ closed and there is a c.e.\ set of open rational balls, each which intersects $N_{\mathrm{red}}$, such that the set contains all balls centered at rational points belonging to $N_{\mathrm{red}}$.
  \end{enumerate}
\end{theorem}
\begin{proof}
  Let $\vk{g}$ be such that $G^{\dagger} = (G,\vk{g})$.
  Let $\vk{c}$ be the corresponding sequence of generators for $\contr{\mb{K}}{|\vk{g}|}$.
  Let $e$ be the identity in $G$.

  (a) $\implies$ (b).
  Since $G^{\dagger}$ is c.e., $\ugcsa{\mb{K}}{G^{\dagger}}$ is c.e..
  By \autoref{c.e. wp}, $N_{\mathrm{uni}}$ is c.e.\ closed.

  We show $N_{\mathrm{red}}$ is strongly co-c.e.\ closed.
  We enumerate through all rational points $r(\vk{g})$ of $\rgcsa{\mb{K}}{G^{\dagger}}$ and all positive dyadic rationals $d$.
  Since $G^{\dagger}$ is computable, for each pair $r(\vk{g})$ and $d$, we can begin an enumeration of all finite sums $\sum_{i=1}^{n} a_{i}h_{i}$ where $a_{1},\ldots,a_{n}$ are nonzero rationals of $\mb{K}$, $h_{1},\ldots,h_{n}$ are words on $\vk{g}$ such that $h_{i} \neq h_{j}$ in $G$ for $1 \leq i < j < n$, and \[\norm{\sum_{i=1}^{n} a_{i}h_{i}}_{\ell^{2}(G)}^{2} = \sum_{i=1}^{n} a_{i}^{2} \leq 1.\]
  Again using that $G^{\dagger}$ is computable, we can effectively rewrite the action of $r(\vk{g})$ on $\sum_{i=1}^{n} a_{i}h_{i}$ into the form $\sum_{i=1}^{m} b_{i}f_{i}$ where $b_{1},\ldots,b_{m}$ are nonzero rationals of $\mb{K}$, and $f_{1},\ldots,f_{m}$ are words on $\vk{g}$ such that $f_{i} \neq f_{j}$ in $G$ for $1 \leq i < j \leq m$.
  Note \[\norm{r(\vk{g})}_{\mathrm{red}} \geq \norm{\sum_{i=1}^{m} b_{i}f_{i}}_{\ell^{2}(G)} = \bigg(\sum_{i=1}^{m} b_{i}^{2}\bigg)^{1/2}.\]
  If $d < (\sum_{i=1}^{m} b_{i}^{2})^{1/2}$, we output $\widehat{B}(r(\vk{c}),d)$.
  Since $\norm{r(\vk{g})}_{\mathrm{red}}$ is the supremum of all such $\norm{\sum_{i=1}^{m} b_{i}f_{i}}_{\ell^{2}(G)}$, by \autoref{norm lifts}, we have shown $N_{\mathrm{red}}$ is strongly co-c.e.\ closed.


  (b) $\implies$ (c).
  Since $N_{\mathrm{red}}$ is strongly co-c.e.\ closed, we can enumerate all rational points $r$ such that $\norm{r}_{\mathrm{red}} > 0$.
  Now, we just observe that $\norm{r}_{\mathrm{uni}} > 0$ if and only if $\norm{r}_{\mathrm{red}} > 0$ for rational points $r$.

  (c) $\implies$ (a).
  Let $w$ be a word on generators $\vk{g}$.
  Then $\norm{w - e}_{\mathrm{uni}} = 0$ or $\norm{w - e}_{\mathrm{uni}} \geq \sqrt{2}$.
  Since $N_{\mathrm{uni}}$ is c.e.\ closed, we can enumerate all open rational balls $B(r(\vk{c}),d)$ that intersect $N_{\mathrm{uni}}$.
  If ever $r(\vk{g}) = w - e$ and $d < \sqrt{2}$, we have determined that $w = e$ in $G$.
  Simultaneously, we enumerate all rational points $q(\vk{c})$ which belong $N_{\mathrm{uni}}^{c}$.
  If ever $q(\vk{g}) = w - e$, then $\norm{w - e}_{\mathrm{uni}} > 0$, so $w \neq e$ in $G$.

  (b) $\implies$ (d).
  Since $N_{\mathrm{uni}}$ is c.e.\ closed, we can enumerate all open rational balls $B(r(\vk{c}),d)$ such that $\norm{r(\vk{g})}_{\mathrm{uni}} < d$.
  Any such open rational ball intersects $N_{\mathrm{red}}$ since $\norm{\cdot}_{\mathrm{red}} \leq \norm{\cdot}_{\mathrm{uni}}$.
  If $r(\vk{c})$ belongs to $N_{\mathrm{red}}$, then $\norm{r(\vk{g})}_{\mathrm{red}} = 0$, so $\norm{r(\vk{g})}_{\mathrm{uni}} = 0$.
  Hence $B(r(\vk{c}),d)$ will be enumerated for all positive dyadic rationals $d$.

  (d) $\implies$ (a).
  Let $w$ be a word on generators $\vk{g}$.
  Then $\norm{w - e}_{\mathrm{red}} = 0$ or $\norm{w - e}_{\mathrm{red}} \geq \sqrt{2}$.
  We express $w - e$ as a rational point $q(\vk{g})$.
  Since $N_{\mathrm{red}}$ is co-c.e.\ closed, we can enumerate a sequence of open rational balls $B(r(\vk{c}), d)$ whose union is $N_{\mathrm{red}}^{c}$.
  If ever $\norm{q(\vk{c}) - r(\vk{c})} < d$, then $q(\vk{c})$ belongs to $N_{\mathrm{red}}^{c}$, so we have determined $w \neq e$ in $G$.
  Simultaneously, we enumerate a sequence of open rational balls $B(s(\vk{c}),e)$, each which intersects $N_{\mathrm{red}}$, such that the sequence contains all balls centered at rational points belonging to $N_{\mathrm{red}}$.
  If ever $s(\vk{g}) = w - e$ and $d < \sqrt{2}$, we have determined that $w = e$ in $G$.
\end{proof}
The properties (c) and (d) may fail to be computably robust for an arbitrary C*-algebra.
However, in theorem above, the properties are robust in the sense that they are preserved between computably isomorphic presentations of $G$.

Observe if $\ugcsa{\mb{K}}{G^{\dagger}}$ has computable closed word problem, then (c) is satisfied, so $G^{\dagger}$ is computable.
Similarly, if $\rgcsa{\mb{K}}{G^{\dagger}}$ has computable closed word problem,  then (d) is satisfied, so $G^{\dagger}$ is computable.

\begin{question}
  Is there a computable presentation $G^{\dagger}$ of a countable discrete group $G$ such that $\ugcsa{\mb{K}}{G^{\dagger}}$ or $\rgcsa{\mb{K}}{G^{\dagger}}$ does not have computable closed word problem?
\end{question}

If $G$ is amenable, then $\ugcsa{\mb{K}}{G^{\dagger}} = \rgcsa{\mb{K}}{G^{\dagger}}$, see~\cite{Davidson_AlgebrasExample_1996}, so we have the following stronger characterization.
\begin{corollary}
  Let $G$ be an amenable discrete group with presentation $G^{\dagger}$.
  Then $G^{\dagger}$ is computable if and only if $\ugcsa{\mb{K}}{G^{\dagger}} = \rgcsa{\mb{K}}{G^{\dagger}}$ is computable.
\end{corollary}

For finitely generated groups, we can thus restate the theorem as follows.
\begin{corollary}
Let $G$ be a finitely generated discrete group.
Then $G$ has solvable word problem if and only if the word problem of $\ugcsa{\mb{K}}{G^{\dagger}}$ is c.e.\ closed and the word problem of $\rgcsa{\mb{K}}{G^{\dagger}}$ is strongly co-c.e.\ closed for all (for some) presentations $G^{\dagger}$ of $G$.
If in addition $G$ is amenable, then $G$ has solvable word problem if and only if $\ugcsa{\mb{K}}{G^{\dagger}} = \rgcsa{\mb{K}}{G^{\dagger}}$ is computable for all (for some) presentations $G^{\dagger}$ of $G$.
\end{corollary}

By Boone~\cite{Boone_WordProblem_1958} and Novikov~\cite{Novikov_AlgorithmicUnsolvability_1955}, there are finitely presented groups with unsolvable word problem.
If $G$ is such a group with corresponding presentation $G^{\dagger}$, then $\ugcsa{\mb{K}}{G^{\dagger}}$ is finitely c.e.\ but not computable, in fact, the word problem is not even computable closed.



\section{The relationship between real and complex C*-algebras}\label{sec4}

In this section, we investigate the relationship between presentations of real C*-algebras and their complexifications.
The main benefit is the ability to extend results already established for real Banach algebras to real and complex C*-algebras.
In particular, using a result of Melnikov and Ng~\cite{MelnikovNg_ComputableStructures_2016}, we show $\cts{\mb{K}}{[0,1]}$ is not computably categorical as a C*-algebra over $\mb{K}$.

First, we include a result about presentations of abelian C*-algebras $\ctsvai{\mb{K}}{X}$ induced from presentations of $X$.
By \autoref{accsa}, every abelian complex C*-algebra is of this form.
However, by \autoref{arcsa}, only those abelian real C*-algebras with trivial $*$-operation are of this form.

\begin{definition}
Given a separable metric space $X$, a \inlinedef{presentation} of $X$ is a pair $(X,\vk{x})$ where $\vk{x}$ is a countable sequence of elements from $X$ which is dense in $X$.
The presentation $(X,\vk{x})$ is \inlinedef{computable} if $d(x_{i}, x_{j})$ is computable uniformly in $i$ and $j$.
\end{definition}

We will only really be concerned with proper metric spaces, i.e., metric spaces in which every closed ball is compact. Note proper metric spaces are both locally compact and complete.

\begin{definition}
  A computable presentation $X^{\dagger}$ of a proper metric space $X$ is \inlinedef{computably proper} if there is an effective procedure which, when given a closed rational ball $K$ and $i \in \mb{N}$, returns a finite sequence of open rational balls of radius at most $2^{-i}$ that covers $K$.
\end{definition}

For those in computable analysis, this is simply a reformulation of the effective covering property for metric spaces with compact closed balls, see~\cite{Iljazovic_ChainableCircularly_2009}.
Furthermore, if $X$ is compact then a computable presentation of $X$ is computably proper if and only if it is effectively compact as defined in~\cite{IljazovicKihara_ComputabilitySubsets_2021}.

We extend an observation made by Tim McNicholl in~\cite{BazhenovEtal_ComputableStone_2021} and include a proof.
\begin{theorem}\label{cpt proper implies cpt}
  Let $X$ be a separable proper metric space which admits a computably proper presentation.
  Then $\ctsvai{\mb{K}}{X}$ admits a computable presentation as a C*-algebra over $\mb{K}$.
  In particular, $\ctsvai{\mb{R}}{X}$ admits a computable presentation as a real Banach space.
\end{theorem}
\begin{proof}
  Let $X^{\dagger} = (X,{(x_n)}_{n \in \mb{N}})$ be a computably proper presentation of $X$.
  For each $n \in \mb{N}$, let $f_n : X \to \mb{K}$ be defined by $f_n(z) = \frac{1}{1 + d(x_n,z)}$, and observe $f_n$ belongs to $\ctsvai{\mb{K}}{X}$.
  Since ${(f_n)}_{n \in \mb{N}}$ separates points and vanishes nowhere, $\cgen{\mb{K}}{{(f_n)}_{n \in \mb{N}}} = C_0(X; \mb{K})$ by Stone-Weierstrass.
  Let $\ctsvai{\mb{K}}{X}^{\dagger} = (\ctsvai{\mb{K}}{X}, {(f_{n})}_{n \in \mb{N}})$.

  We show $\ctsvai{\mb{K}}{X}^{\dagger}$ is computable.
  We are given a rational point $q(f_{\ell_1},\ldots,f_{\ell_n})$ of $\ctsvai{\mb{K}}{X}^{\dagger}$ and a positive integer $k \in \mb{N}$.
  We must compute a rational $r$ such that $|\norm{q(f_{\ell_1},\ldots,f_{\ell_n})} - r| < 2^{-k}$.
  By \autoref{moc}, we can effectively determine $j \in \mb{N}$ so that, for all $z,w \in X$, if $\max_i |f_{\ell_i}(z) - f_{\ell_i}(w)| \leq 2 ^{-j}$, then $|q(f_{\ell_1}(z),\ldots,f_{\ell_n}(z)) - q(f_{\ell_1}(w),\ldots,f_{\ell_n}(w))| < 2^{-(k+1)}$.
  Let $M$ a positive integer which bounds the sum of the absolute values of the coefficients of $q$.
  Let $K = \bigcup_{i=1}^{n}\{z \in X \st d(x_{\ell_i}, z) \leq M2^{k+1}\}$, so $K$ is a finite union of closed rational balls.
  As $X^{\dagger}$ is computably proper, we can effectively determine a sequence of open rational balls of radius at most $2^{-j}$, centered at points $x_{\nu_{1}},\ldots,x_{\nu_{r}}$, such that the sequence covers $K$.
  Let $N = \max_i |q(f_{\ell_1}(x_{\nu_i}),\ldots,f_{\ell_n}(x_{\nu_i}))|$, so $N$ is computable since $X^{\dagger}$ is computable.

  Now, for all $z \in X$, either there exists $x_{\nu_{m}}$ such that $d(x_{\nu_{m}}, z) < 2^{-j}$, or $z \not\in K$ so $|q(f_{\ell_1}(z),\ldots,f_{\ell_n}(z))| \leq 2^{-(k+1)}$.
  If $d(x_{\nu_m}, z) < 2^{-j}$, then since \[\max_i |f_{\ell_i}(x_{\nu_m}) - f_{\ell_i}(z)| = \max_{i}\bigg|\frac{d(x_{\ell_{i}}, z) - d(x_{\ell_{i}}, x_{\nu_{m}})}{(1 + d(x_{\ell_{i}}, x_{\nu_{m}}))(1 + d(x_{\ell_{i}}, z))}\bigg| \leq d(x_{\nu_m},z),\]
  we can conclude $|q(f_{\ell_1}(z),\ldots,f_{\ell_n}(z))| < N + 2^{-(k+1)}$.
  Hence \[N \leq \norm{q(f_{\ell_1},\ldots,f_{\ell_n})} \leq N + 2^{-(k+1)}.\]
  Since $N$ is computable, we can effectively determine a rational $r$ such that $|\norm{q(f_{\ell_1},\ldots,f_{\ell_n})} - r| < 2^{-k}$, as required.

  If $C_0(X;\mb{R})$ admits a computable presentation as a real C*-algebra (real Banach algebra), then the real Banach space presentation formed from a computable sequence of the products of the special points is computable.
\end{proof}

Now, we begin our investigation of the relationship between real C*-algebras and their complexifications.

We can easily extend our computability notions from complex C*-algebras to complex C*-algebras with an associated conjugation operation.
This is, once again, an instance of the framework for arbitrary metric structures developed in~\cite{FranklinMcnicholl_DegreesLowness_2020}.
Let $B$ be a complex C*-algebra and let $\tau$ be a conjugation on $B$.
We say $(B,\tau,(z_{n})_{n \in \mb{N}})$ is a \inlinedef{presentation} of $(B,\tau)$ if $\{z_{n} \st n \in \mb{N}\} \cup \{\tau(z_{n}) \st n \in \mb{N}\}$ generates $B$ as a complex C*-algebra.
The \inlinedef{rational points} of $(B,\tau,(z_{n})_{n \in \mb{N}})$ take the form $p(z_{i_{1}},\ldots,z_{i_{j}};\tau(z_{i_{j+1}}),\ldots,\tau(z_{i_{k}}))$ for $i_{1},\ldots,i_{k} \in \mb{N}$ where $p$ is a rational $*$-polynomial in $k$ noncommuting variables with no constant term.
We can then define \inlinedef{computable presentations} and \inlinedef{computable categoricity} as we did for complex C*-algebras, where we now require isomorphisms to preserve $\tau$.

\begin{lemma}\label{complexification cc}
  Let $B$ be a complex C*-algebra and let $\tau$ be a conjugation on $B$.
  Let $A = \{b \in B \st \tau(b) = b\}$ be the real C*-algebra determined by $(B,\tau)$.
  Then any computable presentation of $(B,\tau)$ induces a computable presentation of $A$.
  Furthermore, if $A$ is computably categorical as a real C*-algebra, then $(B,\tau)$ is computably categorical as a complex C*-algebra with conjugation.
\end{lemma}
\begin{proof}
If $(B, \tau, (z_n)_{n \in \mb{N}})$ is a presentation of $(B,\tau)$, then we consider the induced presentation $(A, (x_n)_{n \in \mb{N}})$ on $A$ given by $x_{2n-1} = \frac{1}{2}(z_n + \tau(z_n))$ and $x_{2n} = \frac{1}{2i}(z_n - \tau(z_n))$ for $n \in \mb{N}$.
Let $(B,\tau)^{+}$ and $(B,\tau)^\dagger$ be computable presentations of $(B,\tau)$.
Let $A^{+}$ and $A^\dagger$ be the presentations of $A$ induced by $(B,\tau)^{+}$ and $(B,\tau)^\dagger$ respectively.
Then $A^{+}$ and $A^\dagger$ are also computable.
As $A$ is computably categorical, there exists a computable isomorphism $\eta: A^{+} \to A^\dagger$.
Thus $\varphi : B^{+} \to B^\dagger$, defined by $\varphi(z) = \eta\big(\frac{1}{2}(z + \tau(z))\big) + i\eta\big(\frac{1}{2i}(z - \tau(z))\big)$, is a computable isomorphism.
\end{proof}

If $B$ is abelian, then we can consider the $*$-operation as a conjugation on $B$, and in this case, there is a strong converse.

\begin{theorem}\label{cc iff sa cc}
  Let $B$ be an abelian complex C*-algebra with a computable presentation.
  Let $A$ be the subset of self-adjoint elements of $B$.
  Then $A$ is a real C*-algebra, and any computable presentation of $A$ induces a computable presentation of $B$.
  Furthermore, the following are equivalent:
  \begin{enumerate}[label= (\alph*)]
  \item $B$ is computably categorical as a complex C*-algebra,
  \item $A$ is computably categorical as a real C*-algebra,
  \item $A$ is computably categorical as a real Banach algebra.
  \end{enumerate}
\end{theorem}
\begin{proof}
  Note $A$ is the real C*-algebra determined by $(B,*)$.
  If $(A, (x_n)_{n \in \mb{N}})$ is a presentation of $A$, then we call $(B,(x_n)_{n \in \mb{N}})$ the induced presentation of $B$.
  Furthermore, if $(A,(x_n)_{n \in \mb{N}})$ is computable, then $(B,(x_n)_{n \in \mb{N}})$ is computable since \[\norm{r + is} =\sqrt{\norm{(r+is)^*(r + is)}} = \sqrt{\norm{r^2 + s^2}}\] for all rational points $r$ and $s$ of $(A,(x_n)_{n \in \mb{N}})$.

  (b) $\iff$ (c).
  This follows since $* : A \to A$ is simply the identity map.

  (b) $\implies$ (a).
  By \autoref{complexification cc}, if $A$ is computably categorical as a real C*-algebra, then $(B,*)$ is computably categorical as a complex C*-algebra with conjugation.
  Since the conjugation on $B$ is just the $*$-operation, $B$ is computably categorical as a complex C*-algebra.

  (a) $\implies$ (b)
  Let $A^{+}$ and $A^\dagger$ be computable presentations of $A$.
    Let $B^{+}$ and $B^\dagger$ be the computable presentations on $B$ induced by $A^{+}$ and $A^\dagger$ respectively.
  As $B$ is computably categorical, there exists a computable isomorphism $\varphi : B^{+} \to B^\dagger$.
  Then $\varphi \restriction A : A^{+} \to A^\dagger$ is a computable isomorphism.
\end{proof}

We can achieve a partial converse to~\autoref{cpt proper implies cpt} by extending the result, in~\cite{BazhenovEtal_ComputableStone_2021}, which states that a separable Stone space $Z$ is computably metrizable if and only if $\cts{\mb{R}}{Z}$ has a computable presentation as a real Banach space.
Here, a \inlinedef{Stone space} is a totally disconnected compact Hausdorff space, and we say a separable Stone space $Z$ is \inlinedef{computably metrizable} if it admits a metric $d$ such that $(Z,d)$ has a computable presentation.

\begin{corollary}
  Let $Z$ be a separable Stone space.
  Then the following are equivalent:
  \begin{enumerate}[label= (\alph*)]
  \item $Z$ is computably metrizable,
  \item $\cts{\mb{R}}{Z}$ has a computable presentation as a real Banach space,
  \item $\cts{\mb{R}}{Z}$ has a computable presentation as a real C*-algebra,
  \item $\cts{\mb{C}}{Z}$ has a computable presentation as a complex C*-algebra.
  \end{enumerate}
\end{corollary}
\begin{proof}
  (a) $\iff$ (b) is the result in~\cite{BazhenovEtal_ComputableStone_2021}.

  (c) $\iff$ (d) was established above.

  As shown in~\cite{Harrison-trainorEtal_ComputabilityPolish_2020}, every computably metrizable Stone space has a computably compact presentation.
  By \autoref{cpt proper implies cpt}, (a) $\implies$ (d).

  Finally, (c) $\implies$ (b) follows by taking any real C*-algebra presentation which is computable and considering a computable sequence of the products of the special points.
\end{proof}

We also have the following.
\begin{corollary}\label{c01 not cc}
  $\cts{\mb{K}}{[0,1]}$ is not computably categorical as a C*-algebra over $\mb{K}$.
\end{corollary}
\begin{proof}
  By \autoref{cc iff sa cc}, $\cts{\mb{K}}{[0,1]}$ is computably categorical as a C*-algebra over $\mb{K}$ if and only if $\cts{\mb{R}}{[0,1]}$ is computably categorical as a real Banach algebra.
  Melnikov and Ng showed in~\cite{MelnikovNg_ComputableStructures_2016} that $\cts{\mb{R}}{[0,1]}$ is not computably categorical in the language of real Banach algebras.
\end{proof}
This lies in contrast to the group situation, where every finitely generated group that admits a computable presentation is computably categorical (see~\cite{Maltsev_ConstructiveAlgebras_1961}), since $\cts{\mb{K}}{[0,1]}$ is finitely generated with a computable presentation, even one that is finitely c.e.\ as witnessed by \[\cts{\mb{K}}{[0,1]} \cong \urep{\mb{K}}{1,x}{\Iden(1;x),x=x^{*},\norm{x} \leq 1, \norm{1 - x} \leq 1},\] but is not computably categorical.

\section{Computable categoricity of finite-dimensional C*-algebras}\label{sec5}

Although \autoref{c01 not cc} introduces a divergence between computable categoricity results for C*-algebras and for groups, in this section, we establish that finite-dimensional C*-algebras are, as one would expect, computably categorical.

The following is folklore for computable unital presentations, but since we may not have a unit we have to work a little harder.
\begin{lemma}\label{lem: computable point computable spectrum}
  Let $A$ be a C*-algebra over $\mb{K}$ with computable presentation $A^{\dagger}$.
  Let $x$ be a computable point of $A^\dagger$ which is self-adjoint and has finite spectrum $\sigma(x)$.
  Then every element of $\sigma(x)$ is computable.
\end{lemma}
\begin{proof}
  We proceed by induction on the size of $\sigma(x) \setminus \{0\}$.
  Clearly, if $\card(\sigma(x) \setminus \{0\}) = 0$, then $\sigma(x) = \{0\}$ where $0$ is computable.

  Otherwise, let $u \in \sigma(x) \setminus \{0\}$ be such that $|u| = \max_{t \in
    \sigma(x)} |t| = \norm{x}$, so $u$ is computable.
  Let $p(z) = z(z-u) \in \mb{R}[z]$.
  Then $p(x) = x^2 - ux$ is a self-adjoint computable point of $A^\dagger$ and \[\card(\sigma(p(x)) \setminus \{0\}) =
    \card(p[\sigma(x)] \setminus \{0\}) < \card(\sigma(x) \setminus \{0\}).\]
  By the inductive hypothesis, every element of $\sigma(p(x))$ is computable.
  Since $\sigma(p(x)) = p[\sigma(x)]$, every element of $\sigma(x)$ is a real root of the computable polynomial $p(z) - w \in \mb{R}[z]$ for some $w \in \sigma(p(x))$.
  It is well known that the computable reals form a real closed field, see~\cite{Pour-elRichards_ComputabilityAnalysis_2017} for details.
  Thus every element of $\sigma(x)$ is computable.
\end{proof}

With this, we are ready to show computable categoricity for finite-dimensional abelian real and complex C*-algebras.
\begin{theorem}\label{abelian fd rcsalgs are cc}
  Every finite-dimensional abelian real C*-algebra is computably categorical as a real C*-algebra.
\end{theorem}
\begin{proof}
  Let $A$ be an abelian finite-dimensional real C*-algebra, and let $A^\dagger$ be a computable presentation of $A$.
  By \autoref{structure of fd rcsalgs}, we can identify $A$ with $\mb{R}^k \oplus \mb{C}^m$ for some positive integers $k$ and $m$.

  Let $X$ be the set of self-adjoint rational points of $A^\dagger$, and let $Y$ be the set of skew-adjoint rational points of $A^\dagger$, so the set of rational points of $A^\dagger$ is $X + Y$.
  We know $X$ is dense in $\mb{R}^{k}\oplus\mb{R}^{m}$ and $Y$ is dense in $\{0\}^k \oplus {(i\mb{R})}^m$.
  In particular, there must exist $x \in X$ such that $x$ has distinct nonzero entries, and $y \in Y$ such that only the first $k$ entries are zero.
  We view $\mb{R}^{k} \oplus \mb{R}^{m}$ as a subspace of $\blo{\mb{R}}{\mb{R}^{k} \oplus \mb{R}^{m}}$ where $\mb{R}^{k} \oplus \mb{R}^{m}$ acts on itself by multiplication.
  In this sense, the minimal polynomial of $x$ must be of degree $k+m$ with a nonzero constant term.
  Then $x,\ldots,x^{k+m}$ are linearly independent over $\mb{R}$, so $\cgen{\mb{R}}{x} = \mb{R}^{k}\oplus\mb{R}^{m}$.
  Also, we can find $z \in \mb{R}^{k} \oplus \mb{R}^{m}$ such that the last $m$ entries of $zy$ are $i$.
  Hence $\cgen{\mb{R}}{x,y} = \mb{R}^k \oplus \mb{C}^m$.
  By \autoref{cor: isomorphic to finite subset}, $(A,x,y)$ is a computable presentation computably isomorphic to $A^\dagger$ via the identity map.

  The entries of an element $z \in A$ belong to $\sigma_{A}(z)$ since if $a + ib$ is an entry of $z$ for some $a,b \in \mb{R}$, then $(z - a)^{2} + b^{2}$ is not invertible in $A$.
  Every element of $\sigma(x)$ is computable by \autoref{lem: computable point computable spectrum}, so in particular, every entry of $x$ is computable.
  Similarly, $y^2$ is self-adjoint, so by \autoref{lem: computable point computable spectrum}, every element of $\sigma(y^2)$ is computable.
  Then every element of $\sigma(y)$ is computable, since $\sigma(y^2) = \sigma(y)^2$, and the computable complex numbers form an algebraically closed field, see~\cite{Specker_FundamentalTheorem_1990} for details.
  In particular, every entry of $y$ is computable.
  Thus, by \autoref{cor: isomorphic to finite subset}, $(A,x,y)$ is computably isomorphic to the standard presentation $\mb{R}^k \oplus \mb{C}^m$.

\end{proof}

\begin{corollary}\label{thm: abelian fd csalgs are cc}
  Every finite-dimensional abelian complex C*-algebra is computably categorical as a complex C*-algebra.
\end{corollary}
\begin{proof}
  By \autoref{structure of fd ccsalgs}, any abelian finite-dimensional C*-algebra can be identified with $\mb{C}^n$ for some $n \in \mb{N}$.
  It can be observed, with the use of \autoref{cc iff sa cc}, that $\mb{C}^n$ is computably categorical as a complex C*-algebra if and only if $\mb{R}^n$ is computably categorical as real C*-algebra.
\end{proof}

Since $\mb{H}$ is one of the building blocks of finite-dimensional real C*-algebras, but is not abelian, we separately show it is computably categorical.
\begin{lemma}\label{quaternions cc}
  $\mb{H}$ is computably categorical as a real C*-algebra.
\end{lemma}
\begin{proof}
  Let $\mb{H}^\dagger$ be a computable presentation of $\mb{H}$.

  For any nonzero self-adjoint computable point $x$ of $\mb{H}^\dagger$, we know $x \in \mb{R}$, and $\norm{x}$ is computable.
  Hence $1 = \sgn(x)\frac{x}{\norm{x}}$ is a computable point of $\mb{H}^\dagger$.

  We apply the Gram-Schmidt process to a pair of $\mb{R}$-linearly independent skew-adjoint computable points of $\mb{H}^\dagger$, noting that $\langle a,b\rangle_{\mb{R}} = \frac{1}{2}(a^*b + b^*a)$ is a computable operation, to get a pair $p,q$ of $\mb{R}$-orthonormal skew-adjoint computable points of $\mb{H}^\dagger$.

  Then $p^2 = -p^*p = -1 = -q^*q = q^2$, and $0 = \langle p, q \rangle_{\mb{R}} = -\frac{1}{2}(pq + qp)$ so $pq = -qp$.
  Thus there is an automorphism $\varphi$ of $\mb{H}$ which sends $1 \mapsto 1$, $p \mapsto i$, $q \mapsto j$, and $pq \mapsto k$.
  By \autoref{cor: isomorphic to finite subset}, $(\mb{H},1,p,q)$ is a computable presentation computably isomorphic to $\mb{H}^{\dagger}$.
  Therefore, $\varphi$ is a computable isomorphism from $\mb{H}^\dagger$ to the standard presentation $\mb{H}$.
\end{proof}

{\bigskip}
From the rigid characterization of subalgebras generated by a self-adjoint element, we are able to find a finite set of computable minimal projections which spans the set of self-adjoint elements.

Let $A$ be a C*-algebra over $\mb{K}$ with computable presentation $A^{\dagger}$.
Let $p$ be a computable projection in $A^{\dagger}$.
Then $pAp$ is a C*-algebra over $\mb{K}$, and $A^{\dagger}$ induces a presentation $pA^{\dagger}p$ on $pAp$ formed from the products $pz_{1}\cdots z_{n}p$ where $z_{1},\ldots,z_{n}$ are special points of $A^{\dagger}$.
By \autoref{lem: unif comp substructure}, $pA^{\dagger}p$ is a computable presentation of $pAp$ and the inclusion map from $pA^{\dagger}p$ to $A^{\dagger}$ is computable.
\begin{theorem}\label{mp span sa}
  Let $A$ be a finite-dimensional C*-algebra over $\mb{K}$ with computable presentation $A^\dagger$.
  Then there is a finite set $M$ of computable minimal projections in $A^\dagger$ such that $\vspan{\mb{R}}{M}$ is the set of self-adjoint elements of $A$.
\end{theorem}
\begin{proof}
  We proceed by induction on the $\mb{K}$-dimension of $A$.
  Let $X$ be the set of self-adjoint elements of $A$.
  Let $Z$ be the subset of $A^\dagger$-computable points of $X$, so $Z$ is dense in $X$.

  For $z \in Z$, $\cgen{\mb{K}}{z} \cong \mb{K}^n$ for some $n \in \mb{N}$.
  By \autoref{abelian fd rcsalgs are cc} or \autoref{thm: abelian fd csalgs are cc}, $\mb{K}^{n}$ is computably categorical as a C*-algebra over $\mb{K}$.
  In particular, there must be a family $F$ of minimal projections in $\cgen{\mb{K}}{z}$, each computable with respect to $(\cgen{\mb{K}}{z},z)$, such that $z \in \vspan{\mb{R}}{F}$.
  Since every projection in $F$ is computable with respect to $(\cgen{\mb{K}}{z},z)$, they must be also computable with respect to $A^\dagger$ by \autoref{lem: unif comp substructure}.

  If $1$ is a minimal projection in $A$, then $X = \mb{R} = \vspan{\mb{R}}{1}$ and $1$ is $A^{\dagger}$-computable.
  In that case, we just let $M = \{1\}$.

  Suppose $1$ is not a minimal projection in $A$.
  Then there exists $z \in Z$ such that $\cgen{\mb{K}}{z}$ is not isomorphic to $\mb{K}$, so must contain a copy of $\mb{K} \oplus \mb{K}$.
  Hence, there is a nontrivial computable projection $q$ in $A^\dagger$.
  Let $N$ be the set of all nontrivial projections computable with respect to $A^\dagger$.
  If $1$ is $A^\dagger$-computable, then $1 - q \in N$, so $\vspan{\mb{R}}{N} = \vspan{\mb{R}}{N \cup \{1\}} = \vspan{\mb{R}}{Z} = X$.
  If $1$ is not $A^\dagger$-computable, we directly see $\vspan{\mb{R}}{N} = \vspan{\mb{R}}{Z} = X$.

  Let $K$ be a finite subset of $N$ such that $\vspan{\mb{R}}{K} = \vspan{\mb{R}}{N} = X$.
  For any $p \in K$, $pAp$ is a C*-subalgebra of $A$ over $\mb{K}$ of strictly less dimension.
  By the inductive hypothesis, there is a finite set $M_p$ of computable minimal projections in $pA^\dagger p$ such that $\vspan{\mb{R}}{M_p}$ is the set of self-adjoint elements of $pAp$.
  In particular, $p \in \vspan{\mb{R}}{M_p}$.
  Every projection $r$ which is minimal in $pAp$ is also minimal in $A$ since for $t \in A$ if $t \leq r$, then $t \leq p$, so $t \in pAp$.
  Also, each $r \in M_p$ is computable with respect to $A^\dagger$.
  If we let $M = \bigcup_{p \in K} M_p$, then $M$ satisfies the required conditions.
\end{proof}

As an immediate application, we can reduce the computable categoricity of a direct sum to the computable categoricity of its summands.

Let $A$ and $B$ be C*-algebras over $\mb{K}$ with computable presentations $A^{\dagger}$ and $B^{\dagger}$ respectively.
These presentations induce a computable presentation $A^{\dagger}\oplus B^{\dagger}$ on $A \oplus B$ formed from points $(z,0)$ and $(0,w)$ where $z$ is a special point of $A^{\dagger}$ and $w$ is a special point of $B^{\dagger}$.

\begin{theorem}\label{thm: direct sum of ccs is cc}
  Let $A$ and $B$ be finite-dimensional C*-algebras over $\mb{K}$ which are computably
  categorical as C*-algebras over $\mb{K}$.
  Then $A \oplus B$ is computably categorical as a C*-algebra over $\mb{K}$.
\end{theorem}
\begin{proof}

  Let ${(A \oplus B)}^+$ be a computable presentation of $A \oplus B$.
  By \autoref{mp span sa}, there is a finite set $M$ of computable minimal projections in ${(A \oplus B)}^+$ such that $\vspan{\mb{R}}{M}$ is the set of self-adjoint elements of $A \oplus B$.
  Any minimal projection in $A \oplus B$ must belong to either $A$ or $B$ by minimality.
  For $X = A$ or $X = B$, let $M_{X}$ be the subset of $M$ which belongs to $X$.
  Also, let $Z_{X}$ be the set of products $qz_{1} \cdots z_{n}p$ where $z_{1},\ldots,z_{n}$ are special points of $(A \oplus B)^{+}$ and $p,q \in M_{X}$.
  Note $Z_A \subseteq A$ and $Z_B \subseteq B$.
  Since $1 \in \vspan{\mb{R}}{M}$, we have $\cgen{\mb{K}}{Z_A \cup Z_B} = A \oplus B$.
  Hence $\cgen{\mb{K}}{Z_A} = A$ and $\cgen{\mb{K}}{Z_B} = B$.
  Then $(A, Z_A)$ and $(B, Z_B)$ are computable presentations of $A$ and $B$ respectively by \autoref{lem: unif comp substructure}, and $(A,Z_{A}) \oplus (B,Z_{B})$ is a computable presentation computably isomorphic to $(A \oplus B)^{+}$ by \autoref{cor: isomorphic to finite subset}.

  Let $A^{\dagger}$ and $B^{\dagger}$ be computable presentations of $A$ and $B$ respectively.
  Since $A$ is computably categorical, there exists a computable isomorphism $\varphi$ from $A^\dagger$ to $(A, Z_A)$.
  Similarly, since $B$ is computably categorical, there exists a computable isomorphism $\psi$ from $B^\dagger$ to $(B, Z_B)$.
  Then $\varphi \oplus \psi : A^\dagger \oplus B^\dagger \to {(A \oplus B)}^+$ is a computable isomorphism.
\end{proof}

{\bigskip}
Thus, by \autoref{structure of fd ccsalgs} and \autoref{structure of fd rcsalgs}, in order to prove computable categoricity for arbitrary finite-dimensional C*-algebras, it suffices to consider the matrix algebras.
We would like to induct on the dimension of the matrix algebra, but first we need a way to access matrix algebras of smaller dimension.

Let $\mb{D}$ denote $\mb{R}$, $\mb{H}$, or $\mb{C}$ as a C*-algebra over corresponding $\mb{K}$.
\begin{lemma}\label{rec ma}
  Let $\mc{H}$ be a Hilbert space over $\mb{D}$.
  Let ${(p_i)}_{i=1}^n$ be a sequence of minimal projections in $\blo{\mb{D}}{\mc{H}}$.
  If $p_{m+1}$ does not commute with $q_m := \bigvee_{i=1}^m p_i$ for any $1 \leq m < n$, then $\cgen{\mb{K}}{\bigcup_{i=1}^n p_i\blo{\mb{D}}{\mc{H}}p_i} \cong M_n(\mb{D})$.
\end{lemma}
\begin{proof}
  We proceed by induction on the length of the sequence.
  We denote $\blo{\mb{D}}{\mc{H}}$ by $A$ for clarity.

  It can be observed that $\cgen{\mb{K}}{pAp} = pAp = \blo{\mb{D}}{p\mc{H}} \cong \mb{D}$ for any minimal projection $p$ in $A$.

  Now, let ${(p_i)}_{i=1}^{n+1}$ be as stated such that $\cgen{\mb{K}}{\bigcup_{i=1}^n p_i A p_i} \cong M_n(\mb{D})$.
  Since $M_n(\mb{D}) \cong \cgen{\mb{K}}{\bigcup_{i=1}^n p_i A p_i} \subseteq \blo{\mb{D}}{q_n\mc{H}}$, by dimensionality we must have $\kcgen{\bigcup_{i=1}^n p_i A p_i} = \blo{\mb{D}}{q_n\mc{H}}$.
  Then $q_n$ must be of $\mb{D}$-rank $n$, so $q_{n+1}$ is of $\mb{D}$-rank $n+1$ as $p_{n+1}$ does not commute with $q_n$.
  Hence \[\cgen[\bigg]{\mb{K}}{\bigcup_{i=1}^{n+1} p_i A p_i} \subseteq \blo{\mb{D}}{q_{n+1}\mc{H}} \cong M_{n+1}(\mb{D}).\]

  For each $a \in \mb{D}$, we let $R_{a} : \mc{H} \to \mc{H}$ be multiplication on the right by $a$.
  Let $w \in \blo{\mb{K}}{q_{n+1}\mc{H}}$ commute with $\cgen{\mb{K}}{\bigcup_{i=1}^{n+1} p_i A p_i}$.
  In particular, $wq_n$ belongs to the commutant of $\blo{\mb{D}}{q_n\mc{H}}$ in $\blo{\mb{K}}{q_n\mc{H}}$, so is of the form $R_x q_n$ for some $x \in \mb{D}$.
  Also, $wp_{n+1}$ belongs to the commutant of $\blo{\mb{D}}{p_{n+1}\mc{H}} = p_{n+1}Ap_{n+1}$ in $\blo{\mb{K}}{p_{n+1}\mc{H}}$, so is of the form $R_y p_{n+1}$ for some $y \in \mb{D}$.
  As $w$ commutes with $p_{n+1}$, we have that \[R_x p_{n+1}q_n = p_{n+1}R_x q_n = p_{n+1}wq_n = wp_{n+1}q_n = R_y p_{n+1}q_n.\]
  Since $p_{n+1}$ does not commute with $q_n$, $p_{n+1}q_n$ is nonzero, so $x = y$.
  Then $w = R_{x}q_{n+1}$ since $q_{n+1}\mc{H} = p_{n+1}\mc{H} + q_n\mc{H}$.
  Hence $w$ commutes with $\blo{\mb{D}}{q_{n+1}\mc{H}}$.
  Thus, by \autoref{vndct},
  \[\cgen[\bigg]{\mb{K}}{\bigcup_{i=1}^{n+1} p_i A p_i} = \cgen[\bigg]{\mb{K}}{\bigcup_{i=1}^{n+1}p_i A p_i}'' = \blo{\mb{D}}{q_{n+1}\mc{H}}'' = \blo{\mb{D}}{q_{n+1}\mc{H}}.\]
\end{proof}

Now, we are ready to show matrix algebras are computably categorical.

We identify $M_n(\mb{D})$ with the collection of matrices in $M_{n+1}(\mb{D})$ which have all zeroes in their last column and row.
\begin{theorem}
  $M_n(\mb{D})$ is a computably categorical C*-algebra over $\mb{K}$.
\end{theorem}
\begin{proof}
  We proceed by induction on $n$.

  Observe $\mb{D}$ is computably categorical as a C*-algebra over $\mb{K}$ by \autoref{abelian fd rcsalgs are cc}, \autoref{thm: abelian fd csalgs are cc}, or \autoref{quaternions cc}.

  Now, assume $M_{n}(\mb{D})$ is computably categorical.
  For ease of reading, we denote $M_{n+1}(\mb{D})$ by $A$.
  Let $A^\dagger$ be a computable presentation of $A$.

  To make use of the inductive hypothesis, we determine a subalgebra $B$ isomorphic to $M_{n}(\mb{D})$ which is computably included in $A^{\dagger}$.
  By \autoref{mp span sa}, there is a finite set $M$ of computable minimal projections in $A^\dagger$ such that $\vspan{\mb{R}}{M}$ is the set of self-adjoint elements of $A$.
  We construct a sequence $(p_i)_{i=1}^{n+1}$ of minimal projections from $M$ so that $p_{k+1}$ does not commute with $q_k$ for $1 \leq k < n+1$, where $q_k = \bigvee_{i=1}^k p_i$.
  Choose $p_1 \in M$.
  Suppose we have constructed $(p_i)_{i=1}^k$ for some $k \leq n$.
  If $q_k$ commutes with every projection in $M$, then $q_k$ commutes with every self-adjoint element of $A$, so $q_{k} = dI$ for some nonzero $d \in \mb{D}$.
  However, $q_k$ has rank at most $k$, while $dI$ has rank $n+1$.
  Thus, we may choose $p_{k+1} \in M$ which does not commute with $q_k$.
  Let $B = \cgen{\mb{K}}{\bigcup_{i=1}^n p_iAp_i}$.
  Let $B^\dagger$ be the presentation of $B$ formed from the special points of $p_1A^{\dagger}p_{1},\ldots,p_nA^{\dagger}p_{n}$.
  Then $B^{\dagger}$ is computable and the inclusion from $B^{\dagger}$ into $A^{\dagger}$ is computable by \autoref{lem: unif comp substructure}.
  By \autoref{rec ma}, we have $\cgen{\mb{K}}{\bigcup_{i=1}^{n+1} p_iAp_i} \cong M_{n+1}(\mb{D})$ and $B \cong M_n(\mb{D})$.
  Hence $\cgen{\mb{K}}{\bigcup_{i=1}^{n+1} p_iAp_i} = A$ and $B = \blo{\mb{D}}{q_n\mb{D}^{n+1}}$.

  Let $Z$ be the standard orthonormal basis of $\mb{D}$ over $\mb{K}$.  
  By the inductive hypothesis, $M_n(\mb{D})$ is computably categorical over $\mb{K}$.
  In particular, there is a system of matrix units $(f^z_{kj} \st 1 \leq k,j \leq n, z \in Z)$, computable with respect to $B^\dagger$ so computable with respect to $A^\dagger$, such that $B = \cgen{\mb{K}}{\{f^z_{kj} \st 1 \leq k,j \leq n, z \in Z\}}$.

  Let $(e^z_{kj} \st 1 \leq k,j \leq n+1, z \in Z)$ be the standard system of matrix units for $M_{n+1}(\mb{D})$.
  We find a unitary $U$ in $M_{n+1}(\mb{D})$ such that $U^{*}f_{kj}^{z}U = e^{z}_{kj}$ for $1 \leq k,j \leq n$ and $z \in Z$.
  To that end, observe we can extend a $\mb{D}$-orthonormal basis on $(\sum_{k=1}^n f_{kk})(\mb{D}^{n+1})$ to one for $\mb{D}^{n+1}$, so there exists a unitary $V \in M_{n+1}(\mb{D})$ such that $V^*BV = M_n(\mb{D})$.
  Then $(V^*f^z_{kj}V \st 1 \leq k,j \leq n, z\in Z)$ forms a system of matrix units for $M_n(\mb{D})$.
  Before we find $U$, we show there exists a unitary $u$ in $M_{n}(\mb{D})$ such that $u^{*}V^{*}f^{z}_{kj}Vu = e^{z}_{kj}$ for $1 \leq k,j \leq n$ and $z \in Z$.
  There are two cases.
  If the center of $\mb{D}$ is $\mb{K}$, then $u$ exists by \autoref{SN_cor}.
  If the center of $\mb{D}$ is not $\mb{K}$, it must be that $\mb{D} = \mb{C}$, $\mb{K} = \mb{R}$, and $Z = \{1,i\}$.
  Then $\sum_{k=1}^{n} V^*f_{kk}^i V$ belongs to the center of $M_{n}(\mb{C})$, and ${(\sum_{k=1}^{n} V^*f_{kk}^i V)}^2 = -I$, so $\sum_{k=1}^{n} V^*f_{kk}^i V = \pm i$.
  By possibly replacing $f_{kj}^i$ by their negations for $1 \leq k,j \leq n$, we may assume $\sum_{k=1}^{n} f_{kk}^i = i$.
  Then there exists a unitary matrix $u \in M_n(\mb{C})$ such that $u^*V^*f_{kj}Vu = e_{kj}$ for $1 \leq k,j \leq n$, hence $u^*V^*f_{kj}^z Vu = e^z_{kj}$ for $1 \leq k,j \leq n$ and $z \in Z$.
  For both cases, we let $U = V(u + e_{(n+1)(n+1)})$.

  Using $U$, we show we can construct a system of matrix units for $A$ that is computable with respect to $A^{\dagger}$.
  As $U^*p_{n+1}U$ is a projection which does not commute with the identity in $M_n(\mb{D})$, there must exist $1 \leq \ell \leq n$ such that the $(n+1,\ell)$-th entry of $U^*p_{n+1}U$ is nonzero.
  For $z \in Z$, let $w_z = p_{n+1}f_{\ell\ell}^z - \sum_{k=1}^n f_{kk}p_{n+1}f_{\ell\ell}^z$.
  Then there exists a nonzero $d \in \mb{D}$ such that $U^*w_z U  = d e^z_{(n+1)\ell}$ for $z \in Z$.
  Note $\norm{d} = \norm{w_1}$ is computable.
  Now, we are ready to extend the system of matrix units.
  For $z \in Z$, let
  \begin{align*}
    f^z_{(n+1)k} &= \frac{1}{\norm{d}} w_z f_{\ell k} \text{ for } 1 \leq k \leq n,\\
    f^z_{k(n+1)} &= \frac{1}{\norm{d}} f^z_{k\ell}w_1^* \text{ for } 1 \leq k \leq n,\\
    \text{and } f^z_{(n+1)(n+1)} &= \frac{1}{\norm{d}^2}w_z w_1^*.
  \end{align*}

  Then for $z \in Z$, we observe
  \begin{align*}
         U^*f^z_{ij}U &= e^z_{kj} \text{ for } 1 \leq k,j \leq n,\\
         U^*f^z_{(n+1)k}U &= \frac{d}{\norm{d}}e^z_{(n+1)k} \text{ for } 1 \leq k \leq n,\\
         U^*f^z_{k(n+1)}U &= e^z_{k(n+1)}\frac{d^*}{\norm{d}} \text{ for } 1 \leq k \leq n,\\
         \text{and } U^*f^z_{(n+1)(n+1)}U &= \frac{d}{\norm{d}}e^z_{(n+1)(n+1)}\frac{d^*}{\norm{d}}.
  \end{align*}

  We only need to tweak our unitary conjugation.
  Let $x_{1},\ldots,x_{n+1}$ be the standard $\mb{D}$-orthonormal basis for $\mb{D}^{n+1}$.
  If $W \in M_{n+1}(\mb{D})$ is the unitary matrix which sends $x_{i}$ to itself for $1 \leq i \leq n$ and sends $x_{n+1}$ to $x_{n+1}\frac{d}{\norm{d}}$, then ${(UW)}^*f^z_{kj}UW = e^z_{kj}$ for $1 \leq k,j \leq n+1$ and $z \in Z$.
    Note each member of $(f^z_{kj} \st 1 \leq k,j \leq n+1, z \in Z)$ is computable with respect to $A^\dagger$, so by \autoref{cor: isomorphic to finite subset}, $(A, (f^{z}_{kj} \st 1 \leq k,j \leq n+1, z \in Z))$ is computably isomorphic to $A^{\dagger}$ via the identity map.
    Thus conjugation by $UW$ gives a computable isomorphism from $A^\dagger$ to the standard presentation $M_{n+1}(\mb{D})$.
\end{proof}

\begin{corollary}
  Every finite-dimensional C*-algebra over $\mb{K}$ is computably categorical as a C*-algebra over $\mb{K}$.
\end{corollary}

Since the identity is computable with respect to the standard presentation of any finite-dimensional C*-algebra, and automorphisms preserve the identity, we also have the following.
\begin{corollary}
  Let $A$ be a finite-dimensional C*-algebra over $\mb{K}$.
  The identity is computable with respect to any computable presentation of $A$, and $A$ is computably categorical as a unital C*-algebra over $\mb{K}$.
\end{corollary}


\section*{Acknowledgements}
We would like to thank Isaac Goldbring and Tim McNicholl for their truly invaluable feedback on earlier drafts of the manuscript. We would also like to thank the anonymous referee from this journal who suggested many improvements.

\printbibliography[heading=bibintoc,title={References}]

@article{BrattkaPresser_ComputabilitySubsets_2003,
  title = {Computability on Subsets of Metric Spaces},
  author = {Brattka, Vasco and Presser, Gero},
  date = {2003},
  journaltitle = {Theoretical Computer Science},
  shortjournal = {Theoretical Computer Science},
  volume = {305},
  pages = {43--76},
  langid = {english},
  number = {1},
  series = {Topology in {{Computer Science}}}
}

@article{FritzEtal_CanYou_2014,
  title = {Can You Compute the Operator Norm?},
  author = {Fritz, Tobias and Netzer, Tim and Thom, Andreas},
  date = {2014},
  journaltitle = {Proceedings of the American Mathematical Society},
  shortjournal = {Proc. Amer. Math. Soc.},
  volume = {142},
  pages = {4265--4276},
  langid = {english},
  number = {12}
}

@book{Loring_LiftingSolutions_1997,
  title = {Lifting {{Solutions}} to {{Perturbing Problems}} in {{C}}*-Algebras},
  author = {Loring, Terry A.},
  date = {1997},
  publisher = {{American Mathematical Soc.}},
  isbn = {978-0-8218-7191-1},
  langid = {english},
  pagetotal = {180}
}

@book{Davidson_AlgebrasExample_1996,
  title = {C*-Algebras by Example},
  author = {Davidson, Kenneth R.},
  date = {1996},
  publisher = {{American Mathematical Society}},
  location = {{Providence, R.I}},
  isbn = {978-0-8218-0599-2},
  number = {6},
  pagetotal = {309},
  series = {Fields {{Institute}} Monographs}
}

@book{Pour-elRichards_ComputabilityAnalysis_2017,
  title = {Computability in {{Analysis}} and {{Physics}}},
  author = {Pour-El, Marian B and Richards, J. Ian},
  date = {2017},
  isbn = {978-1-107-16844-2},
  langid = {english}
}

@book{Murphy_AlgebrasOperator_2014,
  title = {C*-Algebras and Operator Theory},
  author = {Murphy, Gerald J.},
  date = {2014},
  publisher = {{Academic press}}
}

@article{MelnikovNg_ComputableStructures_2016,
  title = {Computable structures and operations on the space of continuous functions},
  author = {Melnikov, Alexander G. and Ng, Keng Meng},
  date = {2016},
  journaltitle = {Fundamenta Mathematicae},
  volume = {233},
  pages = {101--141},
  publisher = {{Instytut Matematyczny Polskiej Akademii Nauk}},
  langid = {polish}
}

@book{Li_RealOperator_2003,
  title = {Real Operator Algebras},
  author = {Li, Bingren},
  date = {2003},
  publisher = {{World Scientific}},
  location = {{River Edge, N.J}},
  isbn = {978-981-238-380-8},
  pagetotal = {241}
}

@incollection{Specker_FundamentalTheorem_1990,
  title = {The {{Fundamental Theorem}} of {{Algebra}} in {{Recursive Analysis}}},
  booktitle = {Ernst {{Specker Selecta}}},
  author = {Specker, E.},
  editor = {Jäger, Gerhard and Läuchli, Hans and Scarpellini, Bruno and Strassen, Volker},
  date = {1990},
  pages = {264--272},
  publisher = {{Birkhäuser Basel}},
  location = {{Basel}},
  isbn = {978-3-0348-9259-9},
  langid = {english}
}

@online{BazhenovEtal_ComputableStone_2021,
  title = {Computable {{Stone}} Spaces},
  author = {Bazhenov, Nikolay and Harrison-Trainor, Matthew and Melnikov, Alexander},
  date = {2021},
  eprint = {2107.01536},
  eprinttype = {arxiv},
  primaryclass = {math},
  archiveprefix = {arXiv}
}

@article{Harrison-trainorEtal_ComputabilityPolish_2020,
  title = {Computability Of Polish Spaces Up To Homeomorphism},
  author = {Harrison-Trainor, Matthew and Melnikov, Alexander and Ng, Keng Meng},
  date = {2020},
  journaltitle = {The Journal of Symbolic Logic},
  volume = {85},
  number = {4},
  pages = {1664--1686},
  publisher = {{Cambridge University Press}},
  langid = {english}
}

@incollection{IljazovicKihara_ComputabilitySubsets_2021,
  title = {Computability of {{Subsets}} of {{Metric Spaces}}},
  booktitle = {Handbook of {{Computability}} and {{Complexity}} in {{Analysis}}},
  author = {Iljazović, Zvonko and Kihara, Takayuki},
  editor = {Brattka, Vasco and Hertling, Peter},
  date = {2021},
  series = {Theory and {{Applications}} of {{Computability}}},
  pages = {29--69},
  publisher = {{Springer International Publishing}},
  location = {{Cham}},
  isbn = {978-3-030-59234-9},
  langid = {english}
}

@book{Schroder_KtheoryReal_1993,
  title = {K-Theory for Real {{C}}*-Algebras and Applications},
  author = {Schröder, Herbert},
  date = {1993},
  series = {Pitman Research Notes in Mathematics Series},
  number = {290},
  publisher = {{Wiley}},
  location = {{New York}},
  isbn = {978-0-582-21929-8},
  pagetotal = {162}
}

@article{Iljazovic_ChainableCircularly_2009,
  title = {Chainable and {{Circularly Chainable Co-r}}.e. {{Sets}} in {{Computable Metric Spaces}}},
  author = {Iljazovic, Zvonko},
  date = {2009},
  journaltitle = {Journal of Universal Computer Science},
  volume = {15},
  number = {6},
  pages = {1206--1235},
  langid = {english}
}

@book{RordamEtal_IntroductionKtheory_2000,
  title = {An {{Introduction}} to {{K-Theory}} for {{C}}*-{{Algebras}}},
  author = {Rørdam, M. and Larsen, F. and Laustsen, N.},
  date = {2000},
  publisher = {{Cambridge University Press}},
  location = {{Cambridge}},
  isbn = {978-0-521-78944-8}
}

@article{Melnikov_ComputableAbelian_2014,
  title = {Computable Abelian Groups},
  author = {Melnikov, Alexander G.},
  date = {2014},
  journaltitle = {The Bulletin of Symbolic Logic},
  volume = {20},
  number = {3},
  pages = {315--356},
  publisher = {{[Association for Symbolic Logic, Cambridge University Press]}},
}

@article{Kuznetsov_AlgorithmsOperations_1958,
  title = {Algorithms as operations in algebraic systems},
  author = {Kuznetsov, A.V.},
  date = {1958},
  journaltitle = {Uspekhi Matematicheskikh Nauk},
  volume = {13},
  pages = {240--241},
  issue = {3(81)},
  langid = {russian}
}

@article{Blackadar_ShapeTheory_1985,
  title = {Shape Theory for C*-Algebras.},
  author = {Blackadar, Bruce},
  date = {1985},
  journaltitle = {Mathematica Scandinavica},
  volume = {56},
  pages = {249--275},
  langid = {english}
}

@article{ClaninEtal_AnalyticComputable_2019,
  title = {Analytic Computable Structure Theory and $L^p$ Spaces},
  author = {Clanin, Joe and McNicholl, Timothy H. and Stull, Don M.},
  date = {2019},
  journaltitle = {Fundamenta Mathematicae},
  volume = {244},
  pages = {255--285},
  publisher = {{Instytut Matematyczny Polskiej Akademii Nauk}},
  langid = {english}
}

@article{FranklinMcnicholl_DegreesLowness_2020,
  title = {Degrees of and Lowness for Isometric Isomorphism},
  author = {Franklin, Johanna N. Y. and McNicholl, Timothy H.},
  date = {2020},
  journaltitle = {Journal of Logic and Analysis},
  volume = {12},
  langid = {english}
}

@article{YaacovPedersen_ProofCompleteness_2010,
  title = {A Proof of Completeness for Continuous First-Order Logic},
  author = {Yaacov, Itaï Ben And Pedersen, Arthur Paul},
  date = {2010},
  journaltitle = {The Journal of Symbolic Logic},
  volume = {75},
  number = {1},
  pages = {168--190},
  publisher = {{[Association for Symbolic Logic, Cambridge University Press]}},
}

@online{FarahEtal_ModelTheory_2018,
  title = {Model Theory of $\mathrm{C}^*$-Algebras},
  author = {Farah, I. and Hart, B. and Lupini, M. and Robert, L. and Tikuisis, A. and Vignati, A. and Winter, W.},
  date = {2018},
  eprint = {1602.08072},
  eprinttype = {arxiv},
  primaryclass = {math},
  archiveprefix = {arXiv},
}

@article{Cuntz_SimpleAlgebras_1977,
  title = {Simple {{C}}*-Algebras Generated by Isometries},
  author = {Cuntz, Joachim},
  date = {1977},
  journaltitle = {Communications in Mathematical Physics},
  shortjournal = {Commun.Math. Phys.},
  volume = {57},
  number = {2},
  pages = {173--185},
  langid = {english}
}

@article{Boone_WordProblem_1958,
  title = {The {{Word Problem}}},
  author = {Boone, William W.},
  date = {1958},
  journaltitle = {Proceedings of the National Academy of Sciences},
  shortjournal = {PNAS},
  volume = {44},
  number = {10},
  pages = {1061--1065},
  publisher = {{National Academy of Sciences}},
  langid = {english}
}

@article{Novikov_AlgorithmicUnsolvability_1955,
  title = {On the algorithmic unsolvability of the word problem in group theory},
  author = {Novikov, P. S.},
  date = {1955},
  journaltitle = {Trudy Mat. Inst. Steklov.},
  volume = {44},
  pages = {3--143},
  langid = {russian}
}

@article{FarenickPidkowich_SpectralTheorem_2003,
  title = {The Spectral Theorem in Quaternions},
  author = {Farenick, Douglas R. and Pidkowich, Barbara A.F.},
  date = {2003-09},
  journaltitle = {Linear Algebra and its Applications},
  shortjournal = {Linear Algebra and its Applications},
  volume = {371},
  pages = {75--102},
  langid = {english}
}

@book{Lorenz_Algebra_2008,
  title = {Algebra},
  author = {Lorenz, Falko},
  date = {2008},
  volume = {2},
  publisher = {{Springer New York}},
  location = {{New York, NY}},
  isbn = {978-0-387-72487-4},
  langid = {english}
}

@book{Cohn_Algebra_1995,
  title = {Algebra. 2},
  author = {Cohn, Paul M.},
  date = {1995},
  edition = {2. ed., repr. with corr},
  publisher = {{Wiley}},
  location = {{Chichester}},
  isbn = {978-0-471-92234-6},
  langid = {english},
  pagetotal = {428}
}

@article{Maltsev_ConstructiveAlgebras_1961,
  title = {Constructive Algebras I},
  author = {Mal'tsev, A. I.},
  date = {1961-06-30},
  journaltitle = {Russian Mathematical Surveys},
  shortjournal = {Russ. Math. Surv.},
  volume = {16},
  number = {3},
  pages = {77--129},
}

@article{Dyck_GruppentheoretischeStudien_1882,
  title = {Gruppentheoretische Studien},
  author = {Dyck, Walther},
  date = {1882-03},
  journaltitle = {Mathematische Annalen},
  shortjournal = {Math. Ann.},
  volume = {20},
  number = {1},
  pages = {1--44},
  langid = {german}
}

@article{Dyson_WordProblem_1964,
  title = {The Word Problem and Residually Finite Groups},
  author = {Dyson, V.H.},
  date = {1964},
  journaltitle = {Notices Amer. Math. Soc},
  volume = {11},
  number = {78},
  pages = {743}
}

@article{FrohlichShepherdson_EffectiveProcedures_1956,
  title = {Effective Procedures in Field Theory},
  author = {Fröhlich, Albrecht and Shepherdson, J. C.},
  date = {1956-01-26},
  journaltitle = {Philosophical Transactions of the Royal Society of London. Series A, Mathematical and Physical Sciences},
  shortjournal = {Phil. Trans. R. Soc. Lond. A},
  volume = {248},
  number = {950},
  pages = {407--432},
  langid = {english}
}

@article{Maltsev_RecursiveAbelian_1962,
  title = {On Recursive Abelian Groups},
  author = {Mal'tsev, A. I.},
  date = {1962},
  journaltitle = {Dokl. Akad. Nauk SSSR},
  volume = {146},
  number = {5},
  pages = {1009--1012}
}

@article{Melnikov_ComputablyIsometric_2013,
  title = {Computably Isometric Spaces},
  author = {Melnikov, Alexander G.},
  date = {2013},
  journaltitle = {The Journal of Symbolic Logic},
  volume = {78},
  number = {4},
  pages = {1055--1085},
  publisher = {{[Association for Symbolic Logic, Cambridge University Press]}},
}

@article{Rosenberg_StructureApplications_2015,
  title = {Structure and applications of real $C^*$-algebras },
  author = {Rosenberg, Jonathan},
  date = {2016},
  journaltitle = {Contemporary Mathematics},
  volume = {671},
  pages = {235--258},
}

@article{MoslehianEtal_SimilaritiesDifferences_2022,
  title = {Similarities and Differences between Real and Complex {{Banach}} Spaces: An Overview and Recent Developments},
  shorttitle = {Similarities and Differences between Real and Complex {{Banach}} Spaces},
  author = {Moslehian, M. S. and Muñoz-Fernández, G. A. and Peralta, A. M. and Seoane-Sepúlveda, J. B.},
  date = {2022-04},
  journaltitle = {Revista de la Real Academia de Ciencias Exactas, Físicas y Naturales. Serie A. Matemáticas},
  shortjournal = {Rev. Real Acad. Cienc. Exactas Fis. Nat. Ser. A-Mat.},
  volume = {116},
  number = {2},
  pages = {88},
  langid = {english}
}

@incollection{MelnikovNies_ClassificationProblem_2013,
  title = {The {{Classification Problem}} for {{Compact Computable Metric Spaces}}},
  booktitle = {The {{Nature}} of {{Computation}}. {{Logic}}, {{Algorithms}}, {{Applications}}},
  author = {Melnikov, Alexander G. and Nies, André},
  editor = {Bonizzoni, Paola and Brattka, Vasco and Löwe, Benedikt},
  date = {2013},
  series = {Lecture {{Notes}} in {{Computer Science}}},
  volume = {7921},
  pages = {320--328},
  publisher = {{Springer Berlin Heidelberg}},
  location = {{Berlin, Heidelberg}},
  editorb = {Hutchison, David and Kanade, Takeo and Kittler, Josef and Kleinberg, Jon M. and Mattern, Friedemann and Mitchell, John C. and Naor, Moni and Nierstrasz, Oscar and Pandu Rangan, C. and Steffen, Bernhard and Sudan, Madhu and Terzopoulos, Demetri and Tygar, Doug and Vardi, Moshe Y. and Weikum, Gerhard},
  editorbtype = {redactor},
  isbn = {978-3-642-39052-4}
}

@article{GreenbergEtal_UniformProcedures_2018,
  title = {{{UNIFORM PROCEDURES IN UNCOUNTABLE STRUCTURES}}},
  author = {Greenberg, Noam and Melnikov, Alexander G. and Knight, Julia F. and Turetsky, Daniel},
  date = {2018-06},
  journaltitle = {The Journal of Symbolic Logic},
  volume = {83},
  number = {2},
  pages = {529--550},
  publisher = {{Cambridge University Press}},
  langid = {english}
}

@article{Mcnicholl_ComputableCopies_2017,
  title = {Computable Copies of $\ell^p$},
  author = {McNicholl, Timothy H.},
  date = {2017-10-31},
  journaltitle = {Computability},
  shortjournal = {COM},
  volume = {6},
  number = {4},
  pages = {391--408},
}

@incollection{Mcnicholl_NoteComputable_2015,
  title = {A {{Note}} on the {{Computable Categoricity}} of $\ell^p$ {{Spaces}}},
  booktitle = {Evolving {{Computability}}},
  author = {McNicholl, Timothy H.},
  editor = {Beckmann, Arnold and Mitrana, Victor and Soskova, Mariya},
  date = {2015},
  series = {Lecture {{Notes}} in {{Computer Science}}},
  volume = {9136},
  pages = {268--275},
  publisher = {{Springer International Publishing}},
  location = {{Cham}},
  isbn = {978-3-319-20027-9}
}

@article{McnichollStull_IsometryDegree_2019,
  title = {The Isometry Degree of a Computable Copy of $\ell^p$},
  author = {McNicholl, Timothy H. and Stull, Donald M.},
  date = {2019-06-17},
  journaltitle = {Computability},
  shortjournal = {COM},
  volume = {8},
  number = {2},
  pages = {179--189},
}

@article{BrownEtal_ComplexityClassifying_2020,
  title = {{{ON THE COMPLEXITY OF CLASSIFYING LEBESGUE SPACES}}},
  author = {Brown, Tyler A. and Mcnicholl, Timothy H. and Melnikov, Alexander G.},
  date = {2020-09},
  journaltitle = {The Journal of Symbolic Logic},
  volume = {85},
  number = {3},
  pages = {1254--1288},
  publisher = {{Cambridge University Press}},
  langid = {english}
}

@article{BrownMcnicholl_AnalyticComputable_2020,
  title = {Analytic Computable Structure Theory and $L^p$-Spaces Part 2},
  author = {Brown, Tyler and McNicholl, Timothy H.},
  date = {2020-05},
  journaltitle = {Archive for Mathematical Logic},
  shortjournal = {Arch. Math. Logic},
  volume = {59},
  number = {3-4},
  pages = {427--443},
  langid = {english}
}

@article{Mcnicholl_ComputingExponent_2020,
  title = {Computing the Exponent of a {{Lebesgue}} Space},
  author = {McNicholl, Timothy},
  date = {2020-12-29},
  journaltitle = {Journal of Logic and Analysis},
  shortjournal = {J. Log. Anal.},
  volume = {12},
}

@misc{HoyrupEtal_DegreeSpectra_2020a,
  title = {Degree Spectra of Homeomorphism Types of {{Polish}} Spaces},
  author = {Hoyrup, Mathieu and Kihara, Takayuki and Selivanov, Victor},
  date = {2020-04-15},
  number = {arXiv:2004.06872},
  eprint = {2004.06872},
  eprinttype = {arxiv},
  primaryclass = {math},
  publisher = {{arXiv}},
  archiveprefix = {arXiv}
}
\end{document}